\documentclass[a4paper,12pt]{article}
\usepackage[utf8x]{inputenc}
\usepackage{graphicx}
\usepackage[T1]{fontenc}
\usepackage{amsmath}
\usepackage{amsfonts}
\usepackage{amssymb}
\usepackage{amsthm}
\usepackage{amstext}
\usepackage{pstricks}
\usepackage{color}
\usepackage{makeidx}
\makeindex
\usepackage{tocloft}
\usepackage{hyperref}

\newdimen\CdotAxis
\newcommand*{\CdotAux}[3]{%
  {%
    \settoheight\CdotAxis{$#2\vcenter{}$}%
    \sbox0{%
      \raisebox\CdotAxis{%
        \scalebox{#1}{%
          \raisebox{-\CdotAxis}{%
            $\mathsurround=0pt #2#3$%
          }%
        }%
      }%
    }%
    \dp0=0pt %
    \sbox2{$#2\bullet$}%
    \ifdim\ht2<\ht0 %
      \ht0=\ht2 %
    \fi
    \sbox2{$\mathsurround=0pt #2#3$}%
    \hbox to \wd2{\hss\usebox{0}\hss}%
  }%
}

\newtheorem{theorem}{Theorem}[section]
\newtheorem{lemma}[theorem]{Lemma}
\newtheorem{corollary}[theorem]{Corollary}
\newtheorem{proposition}[theorem]{Proposition}

\theoremstyle{definition}

\theoremstyle{remark}

\def \cA {\mathcal{A}}

\def \cD {\mathcal{D}}

\def \cN {\mathcal{N}}

\def \cP {\mathcal{P}}

\def \cS {\mathcal{S}}

\def \cW {\mathcal{W}}

\def \a {\alpha}
\def \b {\beta}
\def \g {\gamma}
\def \d {\delta}
\def \e {\varepsilon}

\def \k {\kappa}
\def \l {\lambda}

\def \L {\Lambda}

\def \A {\mathbb{A}}
\def \B {\mathbb{B}}

\def \dD {\mathbb{D}}

\def \K {\mathbb{K}}

\def \N {\mathbb{N}}

\def \P {\mathbb{P}}

\def \R {\mathbb{R}}

\def \dU {\mathbb{U}}

\def \W {\mathbb{W}}

\def \lra {\longrightarrow}

\def \Lra {\Longrightarrow}

\def \Ot {(O_n)_{n\geq 0}}
\def \Zt {(Z_n)_{n\geq 0}}

\def \pml {\cP^m_{\ell +1}}
\def \zl {\{\, 0,\dots,\ell \,\}}

\def \zk {\{\, 0,\dots,K \,\}}

\def \cW {{\mathcal W}^*}

\def \exa {e^{-a}}

\def\uro{\smash{{U}^{\!\!\!\!\raise5pt\hbox{$\scriptstyle o$}}}}

\def\lmq {{\genfrac{}{}{0pt}{1}{\ell,m\to\infty,\,q\to0}{\ell q\to a,\,m/\ell\to\a}}}
\def\lmqq {{\genfrac{}{}{0pt}{1}{\ell,m\to\infty,\,q\to0}{\ell q\to a}}}
\def\lq {{\genfrac{}{}{0pt}{1}{\ell\to\infty,\,q\to0}{\ell q\to a}}}

\begin{document}

\begin{center}
\begin{LARGE}
The Wright--Fisher model\\[3 pt]
for class--dependent fitness landscapes
\end{LARGE}

\begin{large}
Joseba Dalmau

\vspace{-12pt}
CMAP, \'Ecole Polytecnique

\today

\vspace{4pt}
\end{large}
\end{center}

\begin{abstract}
\noindent
We consider a population evolving under mutation and selection.
The genotype of an individual is a word of length $\ell$
over a finite alphabet.
Mutations occur during reproduction,
independently on each locus;
the fitness depends on the Hamming class (the distance to a 
reference sequence $w^*$).
Evolution is driven according to the classical Wright--Fisher process.
We focus on the proportion of the different classes under the invariant measure of the process.
We consider the regime where the length of the genotypes $\ell$ goes to infinity,
and

\vspace{-0.5 em}
\begin{center}
$\text{population size}\,\sim\,\ell\,,\qquad
\text{mutation rate}\,\sim\,1/\ell\,.$
\end{center}

\vspace{-0.5 em}
\noindent
We prove the existence of a critical curve,
which depends both on the population size and the mutation rate.
Below the critical curve,
the proportion of any fixed class converges to $0$,
whereas above the curve, 
it converges to a positive quantity,
for which we give an explicit formula.
\end{abstract}

\section{Introduction}
Most of the living populations share, among others, 
these three main features:
genomes are long, populations are large, and mutations are rare.
Nevertheless, when modeling a living population,
different relations between those three parameters
will lead to different conclusions.
We focus here on a situation which is most appropriate 
for living beings of small complexity,
as RNA viruses, or replicating macromolecules:
we aim to model a population in which
both the population size and the inverse mutation rate
are of the same order as the length of the genome~\cite{GEFS}.
The main forces that will drive the evolution of such a population
are, of course, mutation, but also selection, and genetic drift.
Selection is introduced via a fitness function on the genotypes,
which encodes the average number of offspring of an individual
carrying a particular genotype.
Genetic drift is introduced by considering a finite population of constant size.
This modeling situation is known to lead to very particular
and interesting phenomena:

\textbf{Error threshold.} 
There is a critical mutation rate separating two different regimes.
Above the critical mutation rate, all genetic information is eventually lost,
while below the critical mutation rate, an equilibrium state is reached
in which the fittest genotype (the master sequence) is present in a positive proportion.

\textbf{Quasispecies.}
The equilibrium that is reached below the error threshold 
consists of a positive proportion of the fittest genotype,
which may be very low, and mutants that are a few mutations away
from the master sequence may appear in high proportions.
Thus, the genetic heterogeneity of such an equilibrium state is huge,
and we might as well not be able to identify the master sequence.
Such a population is often referred to as a quasispecies.

\textbf{Population threshold.}
A low mutation rate is not enough for a quasispecies to form.
Indeed, if the population is too small, it is likely that the master sequences
present in the population mutate all at once or in a few generations,
thus loosing the driving force of the quasispecies. This event becomes
more and more unlikely as the population size grows,
thus giving rise to a second threshold phenomenon, namely a population threshold.

The first two phenomena where first observed by Eigen, 
in a mathematical model for prebiotic populations~\cite{Eigen1}.
The concept of quasispecies was later popularized by Eigen and Schuster~\cite{ES1}.
The model considered by Eigen takes the population size to be infinite,
and models the evolution via a system of differential equations.
The system is studied in the long chain regime, i.e., when 
the length of the genomes goes to infinity, 
and the error threshold and quasispecies phenomena are found. 
In order to observe the population threshold,
it is necessary to consider a model where the population is taken to be finite.
This phenomenon has first been observed in~\cite{CerfM} for the Moran model
and~\cite{CerfWF} for the Wright--Fisher model.
A nice account of the error threshold 
and quasispecies phenomena, the main models where they arise, and their applications
can be found in~\cite{doschu}.
We refer the reader to~\cite{CerfM} for a more detailed exposition
of the different attempts to build finite population models that
present the error threshold and the quasispecies.

Most of the works that show the above three phenomena deal with the simplest possible 
fitness landscape, namely the sharp--peak landscape: there is a single fittest genotype,
the master sequence, and all the other genotypes share the same fitness.
The works~\cite{CerfWF,DWF} show how, in the sharp--peak landscape,
the Wright--Fisher model presents all three of the above phenomena.
Our objective is to extend these results to more general fitness landscapes.
We focus in the present paper on the case of class--dependent fitness functions:
there is a single fittest genotype,
and the fitness of any other genotype is a function of its Hamming distance
to the fittest genotype.
We present the model in section~\ref{model},
while the main result is presented in section~\ref{result},
along with a sketch of the proof.
The remaining sections are devoted to the proof of the main result.

\section{The model}\label{model}
Let $\cA$ be a finite alphabet of cardinality $\k\geq2$,
and let $\ell\geq 1$ represent the length of the genome.
We consider individuals whose genotypes are elements of $\cA^\ell$.
Each genotype $u\in\cA^\ell$ has a fitness $A(u)$
associated to it,
which should be interpreted as
the mean number of children of an individual carrying the genotype $u$.
When a reproduction occurs,
the newborn child is subject to mutations.
We suppose that mutations happen independently 
over each site of the genotype,
with probability $q\in\,]0,1[\,$.
When a particular site mutates, 
the present letter is replaced with 
a uniformly chosen letter from the $\k-1$ remaining ones.
Thus, the probability of mutating from a chain $u$
to another chain $v$ is given by
$$M(u,v)\,=\,\Big(
\frac{q}{\k-1}
\Big)^{d(i,v)}(1-q)^{\ell-d(u,v)}\,,$$
where $d(u,v)$ represents the Hamming distance between $u$ and $v$,
or equivalently, the number of digits the two sequence differ in.

The evolution will be guided by the classical Wright--Fisher process.
Nevertheless, the analysis of the Wright--Fisher process for an arbitrary fitness function $A$
is far too complicated. We focus here on fitness functions of a particular form,
namely the class--dependent fitness functions. We make the following assumptions on $A$.

\textbf{Master sequence.} We assume the existence of a genotype with maximal fitness
$w^*\in\cA^\ell$, which we call the master sequence.

\textbf{Class--dependence.} We assume further that the fitness of a 
genotype $u$ depends only on the number of point mutations away from the master sequence.
All the sequences at Hamming distance $k$ from the master sequence
form the Hamming class $k$, and they all share the same fitness.

\textbf{Eventually constant.}
Finally we assume that there is a Hamming class $K\geq0$ such that 
the fitness of all the genotypes in the classes over $K$ is 1.

Under these assumptions,
we can define a function $A_H:\zl\to\R_+$ such that:

$\bullet$ $A_H(0)>A_H(k)$ for all $1\leq k\leq \ell$.

$\bullet$ For all $u\in\cA^{\ell}$ we have $A(u)=A_H\big(d(u,w^*)\big)$.

$\bullet$ $A_H(K)\neq 1$ and $A_H(k)=1$ for all $K+1\leq k\leq \ell$.

When $K=0$, 
all the genotypes other than the master sequence have fitness 1.
This particular case is referred to as the sharp--peak landscape;
the Wright--Fisher model on the sharp peak landscape
has been studied in detail in~\cite{CerfWF,DWF}.
Our aim is to generalize the results therein
to class--dependent fitness functions which are eventually constant.
One of the main advantages of working with class--dependent
fitness functions is that we can break the space $\cA^\ell$ into Hamming classes.
This is possible because the mutation matrix $M$ respects the Hamming classes (cf.~\cite{CerfWF} for a proof):
fix $0\leq k,l\leq \ell$ and let $X\sim Bin(k,q/(\k-1))$, $Y\sim Bin(\ell-k,q)$
be independent random variables, then for any $u\in\cA^\ell$ in the class $k$,
$$\sum_{v:d(v,w^*)=l}M(u,v)\,=\,P(
k-X+Y=l
)\,.$$
We denote the above quantity by $M_H(k,l)$, and we call $M_H$
the lumped mutation matrix, and $A_H$ the lumped fitness function.
The original mutation matrix $M$ and fitness function $A$ have served their purpose,
and we will not make reference to them again in the rest of the paper.

\textbf{Notation.} In order to ease the notations,
we will no longer add the subscript $H$ to the lumped fitness function
or the lumped mutation matrix, and we will denote them simply by $A$ and $M$.

We consider a population of size $m\geq 1$ evolving according to the classical Wright--Fisher process.
Informally, the transition from the population at time $n$,
to the population at time $n+1$ is done as follows:
$m$ individuals are sampled from the population at time $n$, with replacement.
At each of the $m$ trials, the probability for a given individual to be chosen
is 
$$\frac{\text{fitness of the individual}}{\text{sum of all fitnesses in the population}}\,.$$
Each of the $m$ chosen individuals reproduces, and the offspring mutate.
The ensemble of the $m$ offspring, after mutation, form the population at time $n+1$.
We will only be interested in the proportions of the different Hamming classes,
and not on the distribution of the different genotypes inside the classes themselves;
the only information we actually need about the population at time $n$,
is the number of individuals in each of the Hamming classes.
Indeed, this information is enough to determine the number of individuals in each class at time $n+1$.
The process that keeps this information is
the occupancy process
$\Ot$ and it
will be the starting point of our study.
It is obtained from the original Wright--Fisher process
$(X_n)_{n\geq 0}$
by using a technique known as lumping;
for a formal definition of the original Wright--Fisher process,
as well as for a formal derivation of the occupancy process from it,
we refer the reader to sections 2 and 4 of \cite{CerfWF}.
\color{black}
Let $\pml$ be the set of the ordered partitions 
of the integer $m$ in at most $\ell+1$ parts:
$$
\pml\,=\,
\big\lbrace\,
(o(0),\dots,o(\ell))\in\N^{\ell+1}:
o(0)+\cdots+o(\ell)=m
\,\big\rbrace\,.
$$
A partition $(o(0),\dots,o(\ell))$
is interpreted as an occupancy distribution,
which corresponds to a population with $o(l)$ 
individuals in the Hamming class $l$.
The occupancy process $\Ot$
is a Markov chain with values in $\pml$
and transition matrix given by:
for $ o,o'\in\pml$,
$$
p_O(o,o')\,=\,
\frac{m!}{o'(0)!\cdots o'(\ell)!}
\prod_{0\leq h\leq\ell}\Bigg(
\frac{\sum_{k\in\zl}o(k)A(k)M(k,h)}{\sum_{h\in\zl}o(h)A(h)}
\Bigg)^{o'(h)}\,.
$$
Let $\cS^\ell$ denote the $\ell$--dimensional unit simplex 
$$\cS^\ell\,=\,\big\lbrace\,
x\in\R^{\ell+1}:x_0\geq0,\dots,x_\ell\geq0\ \text{and}\
x_0+\cdots+x_\ell=1
\,\big\rbrace\,.$$
We define the function $F:\cS^{\ell}\lra\cS^{\ell}$ by
setting
$$\forall\, x\in\cS^{\ell}\quad \forall\, k\in\zl\qquad
F_k(x)\,=\,\frac{\displaystyle\sum_{0\leq h\leq\ell}
x_hA(h)M(h,k)}
{\displaystyle 1+\sum_{0\leq h\leq K}
x_h(A(h)-1)}\,.
$$
In view of the expression of the transition matrix,
for all $o\in\pml$ and $n\geq0$,
given that $O_n=o$, the random vector $O_{n+1}$
follows a multinomial law with parameters $m$ and $F(o/m)$.

\textbf{Notation.}
The expression appearing in the denominator of the function $F(x)$
represents the mean fitness of the population $x$.
Since it will recurrently appear in the subsequent formulas,
for any $k\geq K$ and $x\in\R^{k+1}$, we denote 
$$\phi(x)\,=\,1+\sum_{0\leq h\leq K}x_h(A(h)-1)\,.$$

A straightforward treatment of the occupancy process is hardly tractable.
Lucky for us, in most living populations, genomes are long, 
populations large, and mutations rare. 
We will thus carry out the study of the occupancy process in
the following asymptotic regime:
$$\displaylines{
\ell\to +\infty\,,\qquad m\to +\infty\,,\qquad q\to 0\,,\cr
{\ell q} \to a\in\,]0,+\infty[\,\,,
\qquad\frac{m}{\ell}\to\alpha\in\,]0,+\infty[\,\,.}$$
This asymptotic regime has two main consequences
on the normalized occupancy process $(O_n/m)_{n\geq0}$,

$\bullet$ Since $m\to\infty$,
the multinomial law involved in the transition mechanism of the process
concentrates around its mean, which is given by the mapping $F$,
and the trajectories of the process tend to be close to those
of the discrete dynamical system given by the iterates of $F$.

$\bullet$ Since $\ell,1/q\to\infty$ and $\ell q\to a$,
the mutation matrix $M$ converges to an infinite upper diagonal matrix $M_\infty$;
the probability of mutating to a lower class converges to 0,
and the probability of jumping forward converges to a Poisson law of parameter $a$
(cf. the appendix~\ref{Propmut}).

Let $k\geq K$ and define the set 
$\cD^k$ by
$$\cD^k\,=\,
\big\lbrace\,
r\in\R^{k+1}:
r_0\geq 0,\dots,r_k\geq 0 \text{ and } r_0+\cdots+r_k\leq 1
\,\big\rbrace\,.$$
The first $k+1$ coordinates of $F$ 
converge 
to a mapping $G:\cD^k\lra\cD^k$ given by
$$\forall\, r\in\cD^k\quad \forall\, i\in\lbrace\,0,\dots,k\,\rbrace\qquad
G_i(r)\,=\,\phi(r)^{-1}\displaystyle\sum_{h=0}^i r_h A(h)\frac{a^{i-h}}{(i-h)!}\,.$$
Thus, asymptotically,
the coordinates $0,\dots,k$ of the normalized occupancy process
can be seen as a random perturbation
of the discrete dynamical system given by the iterates of $G$:
$$r^0\in\cD\,,\qquad r^n\,=\,G(r^{n-1})\,=\,G^n(r^0)\,,\quad n\geq1\,.$$
In fact, this dynamical system will play a key role in our analysis. 
The mapping $G$ and the dynamical system associated to it
have extendedly been studied in the works~\cite{CD2} and~\cite{Dalmau}.
The main results concerning the fixed points of $G$ are given in~\cite{CD2},
while the results concerning the stability of the fixed points
and the convergence of the dynamical system are given in~\cite{Dalmau}.
We summarize these results in the upcoming propositions.
Consider the set of indexes $b\in\lbrace\,0,\dots,K\,\rbrace$ such that $A(b)\exa>1$ and
$A(b)>A(j)$ for all $j>b$.
Let $I_A$ be the set of these indexes, to which
the index $K+1$ has been added too.
\begin{proposition}
\label{fixed}
The mapping $G$ has as many fixed points in $\cD$ as there are elements in $I_A$.
For each $b\in I_A$,
the associated solution $\rho^b$ is given by 
$\rho^b_0=\cdots=\rho^b_{b-1}=0$ 
and for $0\leq k\leq K-b$,
\begin{multline*}
\rho^b(b+k)\,=\,
\Bigg(\frac{1}{A(b)}+\!\!\!\!\!\!\sum_{\genfrac{}{}{0pt}{1}{h\geq 1}{0=i_0<\cdots<i_h}}
\frac{a^{i_h}}{A(b+i_h)}\prod_{t=1}^h\frac{A(b+i_t)}{(i_t-i_{t-1})!(A(b)-A(b+i_t))}\Bigg)^{-1}
\\
\times\Bigg(\frac{1}{A(b)}1_{k=0}+
\frac{a^k}{A(b+k)}\!\!\!\!\!\!\!\!
\sum_{\genfrac{}{}{0pt}{1}{1\leq h\leq k}{0=i_0<\cdots<i_h=k}}
\prod_{t=1}^h \frac{A(b+i_t)}{(i_t-i_{t-1})!(A(b)-A(b+i_t))}\Bigg)
\,.\end{multline*}
\end{proposition}
Note that in particular, the solution corresponding to the index $K+1$
is identically $0$, and that it is the only fixed point of $G$
if and only if $A(0)\exa\leq1$.
Let $I_A=\lbrace\,b_1,\dots,b_N\,\rbrace$ and note that $N=1$ 
corresponds to $0$ being the only fixed point of $G$.
Define, for $b\in I_A$,
the set $D_b\subset \cD$ by
$$D_b\,=\,\Big\lbrace\,r\in \cD:
r_0+\cdots+r_{b-1}=0,\Big\rbrace\,.$$
We have the following result.
\begin{proposition}\label{convds}
Let $r\in\cD\setminus\lbrace\,0\,\rbrace$. For every $i\in\lbrace\,1,\dots\,N\rbrace$,
$$\lim_{n\to\infty}G^n(r)\,=\,\rho^{b_i}$$
if and only if
$$r_0=\cdots=r_{b_{i-1}}=0\qquad\text{and}\qquad
\max_{b_{i-1}<k\leq b_i}r_k>0\,.$$
Moreover, the map $G$ 
is contracting in a small enough neighborhood of $\rho^b$ intersected with $D_b$.
\end{proposition}
Consider for example the following fitness function:
$$A(0)\,=\,5\,,\quad
A(1)\,=\,2\,,\quad
A(2)\,=\,4\,,\quad
A(3)\,=\,A(4)\,=\,\cdots\,=\,1\,.$$
In this case $K=2$.
Suppose further that $a$ is such that $4\exa>1>2\exa$.
Then, the mapping $G=(G_0,G_1,G_2,)$ has three fixed points 
in the set $\cD=\lbrace\,
r\in\R^3:r_0,r_1,r_2\geq0\ \text{and}\ r_0+r_1+r_2\leq 1
\,\rbrace$.
The point $0$ is always a fixed point,
and in this case, its basin of attraction is just $\lbrace\,0\,\rbrace$.
We have two other fixed points, $\rho^0$ and $\rho^2$.
The basin of attraction of $\rho^2$ is the set 
$\lbrace\,r\in\cD:r_0=0\,\rbrace\setminus\lbrace\,0\,\rbrace$,
and the basin of attraction of $\rho^0$ is the set $\lbrace\,r\in\cD:r_0>0\,\rbrace$.
In fact, if $A(0)\exa>1$, the fixed point 
$\rho^0$ always exists, and its basin of attraction is always the set
$\lbrace\,r\in\cD:r_0>0\,\rbrace$.
Moreover,
the mapping $G$ is contracting in a small enough neighborhood of $\rho^0$.

\section{Main result}\label{result}
Let us denote by $\mu$ 
the invariant probability measure of the 
process $\Ot$.
For any $0\leq k\leq l$ we denote
by $\pi_k$ the mapping $\R^{l+1}\to\R^{k+1}$ that
keeps the first $k+1$ coordinates, i.e.
$$\forall\,x\in\R^{l+1}\qquad \pi_k(x)\,=\,(x_0,\dots,x_k)\,.$$
For $k\geq 0$,
let us denote by $\nu_k$ the image of the measure $\mu$ through 
the mapping $o\mapsto \pi_k(o/m)$.
%
\begin{theorem}\label{main}
There exists a function $\psi:\,]0,+\infty[\,\to[0,+\infty[\,$, which is finite on $\,]0,\ln A(0)[\,$ and vanishes
on $[\ln A(0),+\infty[\,$,
such that:

$\bullet$ If $\a\psi(a)<\ln\k$, then,
for every $k\geq 0$, the measure $\nu_k$
converges weakly to the measure $\d_0$, i.e.,
$$
\lim_\lmq\,\nu_k\to\d_0\,.
$$

$\bullet$ If $\a\psi(a)>\ln\k$, then,
for every $k\geq0$,
the measure $\nu_k$ converges weakly to the measure $\d_{(\rho^0_0,\dots,\rho^0_k)}$, i.e.,
$$
\lim_\lmq\,\nu_k\to\d_{(\rho^0_0,\dots,\rho^0_k)}\,.
$$
\end{theorem}
In terms of the occupancy process $\Ot$, 
the above result can be restated as follows.
\begin{corollary}
We have the following dichotomy:

$\bullet$ If $\a\psi(a)<\ln\k$ then
$$\forall\,k\geq0\qquad
\lim_\lmq\lim_{n\to\infty}\,
E\bigg(
\frac{O_n(k)}{m}
\bigg)\,=\,0\,.$$
$\bullet$ If $\a\psi(a)>\ln\k$ then
$$\forall\,k\geq0\qquad
\lim_\lmq\lim_{n\to\infty}\,
E\bigg(
\frac{O_n(k)}{m}
\bigg)\,=\,\rho^0_k\,.$$
Moreover, in both cases,
$$\forall\,k\geq0\qquad
\lim_\lmq\lim_{n\to\infty}\,
Var\bigg(
\frac{O_n(k)}{m}
\bigg)\,=\,0\,.
$$
\end{corollary}
The remaining section are devoted to the proof of the theorem;
let us now give a sketch of this proof.
Recall that the occupancy process can be seen as a random perturbation
of the discrete--time dynamical system associated to $G$.
When $A(0)\exa\leq 1$, the mapping $G$ has 0 as its only fixed point,
and the result readily follows. 
When $A(0)\exa>1$ the behavior of the process is much more intricate.
First, we need to differentiate between two very different regimes:
the neutral and non--neutral phases.
The neutral phase consists of the populations where none of the classes $0,\dots,K$
are present. The process will then explore the set of those populations 
until it finds an individual in one of the classes $0,\dots,K$.
While exploring the set of the neutral populations, 
there is no selection, and the process behaves the same as if the function
$A$ were constant. In order to study this phase we will rely on the results found in~\cite{CerfWF}
for the sharp--peak landscape, and we will show in section~\ref{nphase} 
that the mean time needed to exit
the neutral phase is of the order of $\k^\ell$.
The non--neutral phase consists of the populations where at least one of the classes
$0,\dots,K$ is present. 
In the set of non--neutral populations, the process $\Ot$ will tend to behave
as the dynamical system associated to $G$, this fact will be rigorously stated
thanks to a large deviations principle, which we develop in section~\ref{LDtrans}.
Inspired by the theory of Freidlin and Wentzell for random perturbations
of dynamical systems~\cite{FW}, we exploit the large deviations principle
in order to control several quantities associated with the process $\Ot$,

$\bullet$ We show that the process is very unlikely to stay away 
from a neighborhood of all of the fixed points for a long time (section~\ref{awayfp}).

$\bullet$ We show that the process enters the basin of attraction of 
the main fixed point $\rho^0$ in a few steps with reasonable probability.
In fact, this is one of the most technical parts of the proof,
since the large deviations principle is of little help. Indeed, we need to
control the probability for the process to create $\eta m$ master sequences
out of $1$ master sequence, for some $\eta>0$ (section~\ref{ems}).

$\bullet$ We estimate the mean time that the process needs to exit 
the set of non--neutral populations, which turns out to be of the order of $e^{m\psi(a)}$.
The function $\psi(a)$ represents the quasipotential linking the points $\rho^0$ and $0$,
or otherwise stated, the ``energy'' of the most likely path the process is to follow
when going from $\rho^0$ to $0$ (section~\ref{Pertime}).

$\bullet$ We show that when inside the set of non--neutral populations,
the process spends most of its time in a neighborhood of $\rho^0$ (section~\ref{Conc}).

Finally, we put all the above estimates together, and
we use them to show the main theorem, with help of the ergodic 
theorem for Markov chains (section~\ref{super} and~\ref{sub}).
The case $K=0$ corresponds to the sharp--peak landscape,
and has been treated in~\cite{CerfWF,DWF}.
The generalization to the class--dependent case is not straightforward.
Indeed, the proofs in~\cite{CerfWF,DWF} rely strongly on 
coupling and monotonicity arguments, which cease to work for arbitrary 
class--dependent functions. In addition, the behavior of 
the dynamical system associated to $G$ is richer;
in the sharp peak landscape, the only possible fixed points 
are $\rho^0$ and $0$, while for more general fitness functions
intermediate fixed points appear.
The new proofs rely on finding estimates that are uniform 
with respect to the initial points, and are therefore more robust
than the original proofs in~\cite{CerfWF,DWF}.

Since
our aim is to send the length of the sequences $\ell$ to infinity,
the number of coordinates of the occupancy process
will grow to infinity with $\ell$. In order to deal with this
inconvenience, we will truncate the process $\Ot$ so that
the number of coordinates is fixed.
Throughout the rest of the section we fix $k$ 
to be an integer larger or equal to $K$.
We define the truncated
process
$\Zt$ by setting
$$\forall\, n\geq 0\qquad
Z_n=\pi_k\big(
O_n
\big)\,.$$
The process $\Zt$ takes values in the set 
$$\dD^k\,=\,\big\lbrace\,
z\in\N^{k+1}:z_0+\cdots+z_k\leq m
\,\big\rbrace\,.$$
The process $\Zt$ is not Markovian, since the coordinates that 
we are leaving out in its definition cannot be ignored when computing
the transition probabilities of the process $\Zt$. Indeed, for any 
$o\in\pml$, $z\in\dD^k$ and $n\geq0$ we have
\begin{multline*}
P\big(
Z_{n+1}=z\,\big|\,O_n=o
\big)\,=\\
\frac{m!}{z_1!\cdots z_k!(m-|z|_1)!}
F_0(o/m)^{z_0}\cdots
F_k(o/m)^{z_k}(1-|\pi_k(F(o/m))|_1)^{m-|z|_1}\,.
\end{multline*}
However, in the asymptotic regime we consider, 
the process $\Zt$ behaves as a small random perturbation of the dynamical system
associated to the mapping $G$, and therefore,
the process $\Zt$ can be seen as being ``asymptotically Markovian''.
We will start by developing a large deviations principle for the transition probabilities 
of the process $\Zt$ in the next section.
In most subsequent sections the process $\Zt$ will be the main
object of our study.
We develop next a large deviations principle for the transition 
probabilities of the process $\Zt$.

\textbf{Notation.}
In defining $\Zt$ we have fixed a coordinate $k\geq K$, but since the treatment is the same for all $k\geq K$,
in the sequel,
we assume that $k=K$.
We will also denote the sets $\dD^K$ and $\cD^K$ simply by $\dD$ and $\cD$,
and the mapping $\pi_K$ by $\pi$.

\section{Large deviations principle}\label{LDtrans}
For
$p,t\in\cD$, 
we define the quantity
$I_K(p,t)$
as follows:
%
$$I_K(p,t)\,=\,
\sum_{k=0}^K t_k\ln\frac{t_k}{p_k}+
(1-|t|_1)
\ln\frac{1-|t|_1}{1-|p|_1}\,,$$
%
%
We make the convention that 
$0\ln0=0\ln(0/0)=0$.
The function
$I_K(p,\cdot)$
is the rate function
governing the large deviations of a multinomial distribution
with parameters $n$ and $p_0,\dots,p_K,1-|p|_1$.
We have the following estimate for the multinomial coefficients:

\begin{lemma}\label{ldmultinom}
Let
$n\geq N\geq1$,
and
$i_1,\dots,i_N\in \N$ be 
such that
$i_1+\cdots+i_N=n$.
We have
$$\Bigg|
\ln\frac{n!}{i_1!\cdots i_N!}
+\sum_{k=1}^N i_k\ln\frac{i_k}{n}
\Bigg|\,\leq\,
N\ln n+2N\,.$$
\end{lemma}
The proof is similar to that of lemma~7.1 of~\cite{CerfWF}.
\color{black}
Thanks to the lemma, for $o\in\pml$ and $z\in\dD$
$$
\ln P\big(
Z_{n+1}=z\,\big|\,
O_n=o
\big)\,=\,-mI_K\Big(\pi\big(F(o/m)\big),z/m\Big)+\Phi(o,z).
$$
The error term
$\Phi(o,z)$
satisfies, for $m$ large enough,
$$\forall\, o\in\pml\quad \forall\, z\in\dD\qquad
\big|\Phi(o,z)\big|\,\leq\, C(K)\ln m\,,$$
where
$C(K)$
is a constant that depends on $K$ but not on $m$.
We define a function
$V_1:\cD\times\cD\to [0,\infty]$
by setting, for
$r,t\in\cD$,
$$V_1(r,t)\,=\,
I_K\big(
G(r),t
\big)
\,.$$ 
For $x\in\cS^\ell$ and $t\in\cD$, we have,
$$\lim_{\genfrac{}{}{0pt}{1}{\ell\to\infty,\,q\to0}{\ell q\to a}}
I_K\Big(\pi\big(F(x) \big),t\Big)\,=\,V_1(\pi(x),t)\,.$$
\textbf{Notation.} For a subset $A$ of $\cD$,
we denote by $\A$ the set $mA\cap\dD$.
For $r\in\R^{K+1}$,
we denote by $\lfloor r \rfloor$
the vector 
$\lfloor r\rfloor=(\lfloor r_0\rfloor,\dots,\lfloor r_K\rfloor)$.
\begin{proposition}\label{pgdtrans1}
The one step transition probabilities of the Markov chain  
$\Zt$
verify the large deviations principle governed by
$V_1$:

$\bullet$ For any subset $U$ of $\cD$ and for any $r\in\cD$,
we have, for $n\geq 0$,
$$
-\inf\big\lbrace\,
V_1(r,t):t\in \uro
\,\big\rbrace\,
\leq\,\liminf_{\genfrac{}{}{0pt}{1}{\ell,m\to\infty,\,q\to 0}{{\ell q} \to a}}
\frac{1}{m}\ln P\big( Z_{n+1}\in \dU \,\big|\, Z_n=\lfloor mr  \rfloor \big)\,.
$$

$\bullet$ For any subsets $U,U'$ of $\cD$,
we have, for $n\geq 0$,
\begin{multline*}
\hspace*{30 pt}\limsup_{\genfrac{}{}{0pt}{1}{\ell,m\to\infty,\,q\to 0}{{\ell q} \to a}}
\frac{1}{m}\ln\, \sup_{z\in \dU} P\big( Z_{n+1}\in \dU' \,\big|\, Z_n=z \big)
\\\leq\,
-\inf\big\lbrace\,
V_1(r,t):
r\in \overline{U},\,
t\in \overline{U}'
\,\big\rbrace\,.\hspace*{40 pt}
\end{multline*}
\end{proposition}
\begin{proof}
We begin by showing the large deviations upper bound.
Let $U,U'$
be two subsets of
$\cD$ and notice that, for all $z\in\dD$ and $n\geq0$
$$P\big(
Z_{n+1}\in \dU'\,\big|\, Z_n=z
\big)\,\leq\,
\sup_{o:\pi(o)=z}P\big(
Z_{n+1}\in  \dU'\,\big|\,O_n=o
\big)\,.$$
Let $o\in\pml$ be such that $\pi(o)\in \dU$.
For 
$n\geq 0$, we have
$$
P\big(Z_{n+1}\in \dU'
\,\big|\, O_n=o\big)\,=\,
\sum_{z'\in \dU'} P\big(
Z_{n+1}=z'\,\big|\, O_n=o
\big)\,.
$$
The number of elements in the sum is of polynomial order in $m$,
the exponent of $m$ depending on $K$ only.
Thus,
thanks to the above estimates on 
the transition probabilities for the process $\Zt$,
we have, for $m$ large enough,
\begin{multline*}
\sup_{o:\pi(o)\in \dU}P\big(Z_{n+1}\in \dU'
\,|\, O_n=o\big)
\,\leq\,
m^{C(K)}\sup_{\genfrac{}{}{0pt}{1}{o:\pi(o)\in \dU}
{z'\in \dU'}}
P\big(
Z_{n+1}=z'\,\big|\,O_n=o
\big)\\
\leq\,m^{C'(K)}\exp\bigg(
-m\min_{\genfrac{}{}{0pt}{1}{o:\pi(o)\in \dU}
{z'\in \dU'}}
I_K\Big(
\pi\big(F(o/m)\big),z'/m
\Big)
\bigg)\,.
\end{multline*}
where
$C(K)$  and $C'(K)$
are constants that depend on
$K$ 
but not on 
$m$.
Define the mappings $\underline{F},\overline{F}:\cD\lra\cD$ by
setting, for all $r\in\cD$ and $k\in\lbrace\,0,\dots,K\,\rbrace$
\begin{align*}
\underline{F}_k(r)\,&=\,\phi(r)^{-1}\sum_{0\leq i\leq K}
r_i A(i)M(i,k)\,,\\
\overline{F}_k(r)\,&=\,\phi(r)^{-1}\bigg(\sum_{0\leq i\leq K}
r_i A(i)M(i,k)+\big(
1-|r|_1 
\big)M(k+1,k)\bigg)\,.
\end{align*}
Asymptotically,
for $0\leq k<j\leq\ell$,
we have $M(j,k)\leq M(k+1,k)$.
Thus, asymptotically, for all $x$ in the unit simplex $\cS^\ell$,
and for all $k\in\zk$,
$$\underline{F}_k\big(
\pi(x)
\big)\,\leq\,F_k(x)\,\leq\,
\overline{F}_k\big(
\pi(x)
\big)\,.$$
Define next the function $\underline{V}:\cD\times\cD\lra[0,+\infty]$
by
$$\forall\, r,t\in\cD\qquad
\underline{V}(r,t)\,=\,\sum_{i=0}^K t_i\ln\frac{t_i}{\overline{F}_i(r)}
+\big(
1-|t|_1\big)
\ln\frac{1-|t|_1}{1-\big|
\pi(\underline{F}(r))
\big|_1}\,.$$
The function $\underline{V}$ satisfies
$$\forall\, x\in\cS^{\ell}\quad \forall\, t\in\cD\qquad
\underline{V}\big(
\pi(x),t
\big)\,\leq\,I_K\big(
\pi(F(x)),t
\big)\,.$$
Moreover, asymptotically, for $r,t\in\cD$,
$$\underline{V}(r,t)\,\lra\,V_1(r,t)\,.$$
Thus,
\begin{multline*}
\sup_{o:\pi(o)\in \dU}P\big(Z_{n+1}\in \dU'
\,|\, O_n=o\big)
\\\leq\,
(m+1)^{C'(K)}\exp\bigg(
-m\min\bigg\lbrace\,
\underline{V}\Big(
\frac{z}{m},\frac{z'}{m}
\Big):z\in \dU,z'\in \dU'
\bigg\rbrace\,
\bigg)\,.
\end{multline*}
For each
$m\geq 1$,
let
$z_m,z'_m\in\dD$,
be two terms that realize the above minimum.
Up to the extraction of a subsequence,
we can suppose that when
$m\to\infty$,
$$\frac{z_m}{m}\to r\in\overline{U}\,,\ \quad
\frac{z'_m}{m}\to t\in\overline{U}'\,.$$
Thus,
$$
\limsup_{\genfrac{}{}{0pt}{1}{\ell,m\to\infty,\,q\to 0}{{\ell q} \to a}} 
-\underline{V}\Big(
\frac{z_m}{m},
\frac{z_m}{m}
\Big)
\,\leq\,
-V_1(r,t)\,.
$$
Optimizing with respect to $r,t$,
we obtain the upper bound of the large deviations principle.
We show next the lower bound.
Let
$r,t\in\cD$
and
notice that,
for all $z\in\dD$ and $n\geq 0$
$$
P\big(
Z_{n+1}=\lfloor mt\rfloor\,\big|\, Z_n=z
\big)\,\geq\,
\inf_{o:\pi(o)=z}P\big(
Z_{n+1}=\lfloor mt\rfloor\,\big|\,O_n=o
\big)\,.
$$
let $o\in\pml$ be such that $\pi(o)=\lfloor m r\rfloor$. 
We have
$$
P\big(
Z_{n+1}=\lfloor mt\rfloor\,\big|\,
O_n=o
\big)
\geq\, m^{-C(K)}\exp\bigg(
-mI_K\Big(
\pi\big(
F(o/m)
\big),\lfloor mt \rfloor/m
\Big)
\bigg),
$$
where $C(K)$ is a constant depending on $K$ but not on $m$.
Define the function $\overline{V}:\cD\times\cD\lra[0,+\infty]$
by
$$
\forall\, r,t\in\cD\qquad
\overline{V}(r,t)\,=\,\sum_{i=0}^K t_i\ln\frac{t_i}{\underline{F}_i(r)}
+\big(
1-|t|_1\big)
\ln\frac{1-|t|_1}{1-\big|
\pi(\overline{F}(r))
\big|_1}\,.$$
The function $\overline{V}$ satisfies
$$\forall\, x\in\cS^{\ell}\quad \forall\, t\in\cD\qquad
\overline{V}\big(
\pi(x),t
\big)\,\geq\,I_K\big(
\pi(F(x)),t
\big)\,.$$
Moreover, asymptotically, for $r,t\in\cD$,
$$\overline{V}(r,t)\,\lra\,V_1(r,t)\,.$$
Thus, for every $o\in\pml$ such that $\pi(o)\,=\,\lfloor mr\rfloor$,
$$
P\big(Z_{n+1}=\lfloor mt\rfloor
\,|\, O_n=o\big)
\,\geq\,
m^{-C(K)}\exp\bigg(
-m
\overline{V}\bigg(
\frac{\lfloor mr\rfloor}{m},\frac{\lfloor mt\rfloor}{m}
\bigg)
\bigg)\,.
$$
We take the logarithm and we send
$m,\ell$ to $\infty$ and $q$ to $0$.
We obtain then
$$
\liminf_{\genfrac{}{}{0pt}{1}{\ell,m\to\infty,\,q\to 0}{{\ell q} \to a}}
\frac{1}{m}\ln P\big(
Z_{n+1}=\lfloor tm \rfloor
\,\big|\,
Z_n=\lfloor r m \rfloor
\big)
\,\geq\,
-V_1(r,t)\,.
$$
Moreover, if
$t\in\uro$,
for $m$ large enough,
$\lfloor tm \rfloor$ belongs to $\dU$.
Therefore,
$$
\liminf_{\genfrac{}{}{0pt}{1}{\ell,m\to\infty,\,q\to 0}{{\ell q} \to a}}
\frac{1}{m}\ln P\big(
Z_{n+1}\in \dU
\,\big|\,
Z_n=\lfloor m r \rfloor
\big)
\,\geq\,
-V_1(r,t)\,.
$$
We optimize over 
$t$ 
and we obtain the large deviations lower bound.
\end{proof}
A similar proof shows that the $l$--step transition probabilities of $\Zt$
also satisfy a large deviations principle.
For $l\geq 2$,
we define a function 
$V_l$ on $\cD\times\cD$ as follows:
$$
V_l(r,t)
\,=\,
\inf\Big\lbrace\,
\sum_{k=0}^{l-1} V_1(s^k,s^{k+1}):\\
s^0=r,\ s^l=t,\ 
s^k\in\cD
\ \text{ for }\ 0\leq k\leq l
\,\Big\rbrace\,.
$$
\begin{corollary}\label{pgdtransl}
For 
$l\geq 1$,
the $l$--step transition probabilities of
$(Z_n)_{n\geq 0}$
satisfy the large deviations principle governed by
$V_l$:

$\bullet$ For any subset $U$ of $\cD$ and for any $r\in\cD$,
we have, for $n\geq 0$,
$$
-\inf\big\lbrace\,
V_l(r,t):t\in \uro
\,\big\rbrace\,
\leq\,\liminf_{\genfrac{}{}{0pt}{1}{\ell,m\to\infty,\,q\to 0}{{\ell q} \to a}}
\frac{1}{m}\ln P\big( Z_{n+l}\in \dU \,\big|\, Z_n=\lfloor mr  \rfloor \big)\,.
$$

$\bullet$ For any subsets $U,U'$ of $\cD$,
we have, for $n\geq 0$,
\begin{multline*}
\hspace*{30 pt}\limsup_{\genfrac{}{}{0pt}{1}{\ell,m\to\infty,\,q\to 0}{{\ell q} \to a}}
\frac{1}{m}\ln \sup_{z\in \dU} P\big( Z_{n+l}\in \dU' \,\big|\, Z_n=z \big)
\\\leq\,
-\inf\big\lbrace\,
V_l(r,t):
r\in \overline{U},\,
t\in \overline{U}'
\,\big\rbrace\,.\hspace*{40 pt}
\end{multline*}
\end{corollary}
The rate function $V_1(r,t)$ is equal to 0
if and only if
$t=G(r)$.
Thus,
the Markov chain
$(Z_n/m)_{n\geq0}$
can be seen as a random perturbation of the dynamical system
associated to the map $G$ (cf. section~\ref{model}).
The next sections study the consequences of the large deviations principles
of proposition~\ref{pgdtrans1} and corollary~\ref{pgdtransl}
on the asymptotic behavior of the process $\Zt$.

\textbf{Notation.}
In the sequel, by ``asymptotically'' we mean:
for $\ell,m$ large enough, $q$ small enough, and $\ell q$ close enough to $a$. All subsequent statements and inequalities need not be true for all
values of $\ell,m$ and $q$, but only asymptotically, even if we do not
state so explicitly.
For $o\in\pml$ or $z\in\dD$, we use the notation
$$E_o(\cdot)\,=\,E(\cdot\,|\,O_0=o)\,,\qquad
E_z(\cdot)\,=\,E_z(\cdot\,|\,Z_0=z)\,.$$
Note that the first expectation will usually be a number,
while the second one will usually be a random variable.
Thus, an expression of the sort
$$P_z(Z_1=0)\geq p\,,$$
should be interpreted as
$$\inf_{o:\pi(o)=z}P_o(Z_1=0)\,\geq p\,.$$

\section{Time spent away from the fixed points}\label{awayfp}
The aim of this section is to show that the process $\Zt$ 
has a small probability of staying away from a neighborhood of the fixed points
for a long time.
We begin by giving a useful lemma.
For a set $A\subset \cD$ and $\e>0$ we denote by $A^\e$ 
the set of points in $\cD$ at a distance smaller than $\e$ from $A$, i.e.,
$$A^\e\,=\,\big\lbrace\,
r\in\cD:d(r,A)<\e
\,\big\rbrace\,.$$
Let $K,U$ be subsets of $\cD$ satisfying :

$\bullet$ The set $K$ is compact and $U$ is open (with respect to the relative topology in $\cD$).

$\bullet$ There exists $\e>0$ such that any trajectory of the dynamical system starting on $K$ goes through $U$,
and does so before exiting $K^\e$, i.e., for all $r\in K$ there exists $n(r)\in\N$ such that
$$G^1(r),\dots,G^{n(r)-1}(r)\in K^\e\quad\text{and}\quad G^{n(r)}(r)\in U\,.$$
\begin{lemma}\label{timefar1}
There exist $h\in\N$
and $c>0$ (depending on $K,U$) such that, asymptotically, for every point $z\in\K$ 
$$P_{z}\Big(
Z_1\not\in \dU,\dots,Z_{h}\not\in \dU
\Big)\,\leq\,
e^{-cm}\,.$$
\end{lemma}
\begin{proof}
For $r\in\R^{K+1}$ and $\eta>0$,
we denote by $B(r,\eta)$ the open ball around $r$ of radius $\eta$ (intersected with $\cD$).
Recall that for $r\in\cD$ we denote by $r^n$ the $n$--th
iterate of $r$ by the map $G$. 
By continuity of the map $G$,
for every $r\in K$,
there exists $h(r)\in\N$ and $0<\eta^r_0,\dots,\eta^r_{h(r)}<\e$ such that,
for all $0\leq n\leq h(r)-1$,
$$
G\big(
B(r^n,\eta^r_n)
\big)\,\subset\,B(r^{n+1},\eta^r_{n+1}/2)\qquad
\text{and}\qquad
B(r^{h(r)},\eta^r_{h(r)})\subset U\,.$$
The family $\lbrace\,
B(r,\eta_0^r):r\in K
\,\rbrace$
forms an open cover of the compact $K$.
Thus, there exist $r_1,\dots,r_M\in K$ such that
$$K\,\subset\, \bigcup_{1\leq i\leq M}
B(r_i,\eta_0^{r_i})\,.$$
Set
$$h\,=\,\max_{1\leq i\leq M} h(r_i)\,.$$
Let $t\in K$ 
and let $i\in\lbrace\,1,\dots,M\,\rbrace$
be such that $t\in B(r_{i},\eta^{r_{i}}_0)$.
We denote the quantity $h(r_i)$ simply by $h(i)$, the open ball $B(r_{i}^n,\eta^{r_{i}}_n)$
by $B_n$,
and we set $\B_n\,=\,mB_n\cap\dD$.
We have then,
\begin{multline*}
\hspace{-6pt}
P_{z}\big(
Z_1\not\in \dU,\dots,Z_{h}\not\in \dU
\big)\,
\leq\,P_{z}\big(
Z_{h(i)}\not\in \dU
\big)\,
=\,1-P_{z}\big(
Z_{h(i)}\in\dU
\big)\\
\leq\,1-P_{z}\big(
Z_1\in \B_1,\dots,Z_{h(i)}\in\B_{h(i)}
\big)
\,=\,P_{z}\big(
Z_n\not\in \B_n
\ \text{for some}\ 1\leq n\leq h(i)
\big)\\
\leq\,
\sum_{1\leq n\leq h(i)}
P_{z}\big(Z_1\in \B_1,\dots,
Z_{n-1}\in \B_{n-1},
Z_n\not\in \B_n
\big)\\
\leq\,
\sum_{1\leq n\leq h(i)}
\sum_{z'\in \B_{n-1}}P_{z}\big(
Z_n\not\in  \B_n\,\big|\,
Z_{n-1}=z
\big)P_{z}\big(
Z_{n-1}=z'
\big)\,.
\end{multline*}
The large deviations principle for the transitions of $(Z_n)_{n\geq0}$
yields the following bound,
\begin{multline*}
\limsup_\lmqq \frac{1}{m}\ln
P\big(
Z_n\not\in  \B_n\,\big|\,
Z_{n-1}=z'
\big)\\
\leq\,
-\inf\big\lbrace\,
V_1(\rho,\rho'):\rho\in B_{n-1}, \rho'\not\in B_n
\,\big\rbrace\,=\,-c_i^{n}\,.
\end{multline*}
Since $G(B_{n-1})\subset B_n$,
the constant $c^{n}_i$ is strictly positive.
Let $0<\eta<c^n_i$.
From the above inequalities, we conclude that
\begin{multline*}
P_{\lfloor tm \rfloor}\big(
Z_1\not\in \dU,\dots,Z_{h}\not\in \dU
\big)\\
\leq\,\exp\big(
-m(c_i^n-\eta)
\big)\sum_{1\leq n\leq h(i)}
P_{z}\big(
Z_{n-1}\in \B_{n-1}
\big)\,\leq\,
h\exp\big(
-m(c^n_i-\eta)
\big)\,.
\end{multline*}
Since $h$ is fixed, 
and since the number of constants $c^n_i$ is finite,
the above probability is bounded by $e^{-mc}$,
for some $c>0$ independent of $t$.
\end{proof}
We have discussed the behavior of the dynamical system associated to the mapping $G$
in section~\ref{model}, we recall that the set $I_A$ encodes the fixed points of $G$.
Let $\d>0$ and define, for $b\in I_A$, the sets
$$U_b\,=\,\big\lbrace\,
r\in\cD\,:\,|r-\rho^{b}|_1<\d
\,\big\rbrace\,,\qquad
U\,=\,\bigcup_{b\in I_A}U_i\,.$$
The set $\cD\setminus U$ is compact and for every $r\in\cD\setminus U$,
$$\lim_{n\to\infty}G^n(r)\,\in\,U\,.$$
We use the lemma to prove the following corollary.
\begin{corollary}\label{timefar2}
There exist $h\in\N$ and $c>0$ such that, asymptotically,
for every $z\in\dD\setminus\dU$ and $n\in\N$, 
$$P_{z}\Big(
Z_t\not\in \dU,0\leq t\leq n
\Big)\,\leq\,
\exp\Big(
-mc\Big\lfloor\frac{n}{h}\Big\rfloor
\Big)\,.$$
\end{corollary}
\begin{proof}
Divide the interval $\lbrace\,0,\dots,n\,\rbrace$
into subintervals of length $h$.
Using iteratively the previous lemma,
we have, for $i\geq 1$,
\begin{multline*}
\!\!\!\!P_{z}\big(
Z_t\not\in \dU,0\leq t\leq (i+1)h
\big)
=\!\!\!\sum_{z'\in\dD\setminus \dU}\!\!
P_{z}\big(
Z_t\not\in \dU,0\leq t\leq (i+1)h,Z_{ih}=z'
\big)\\
=\,\sum_{z'\in\dD\setminus \dU}
P_{z}\big(
Z_t\not\in \dU,0\leq t\leq ih,Z_{ih}=z'
\big)\\
\times P_{z}\big(
Z_t\not\in \dU, i h<t\leq (i+1)h\big|Z_{i h}=z'
\big)
\leq P_{z}\big(
Z_t\not\in \dU,0\leq t\leq ih
\big)e^{-m c}.
\end{multline*}
Iterating this procedure we get
$$P_{z}\big(
Z_t\not\in \dU,0\leq t\leq (i+1)h
\big)\,\leq\,
e^{-mc(i+1)}\,.$$
Taking $i+1=\lfloor n/h\rfloor$ gives the desired result. 
\end{proof}

\section{Creating enough master sequences}\label{ems}
Throughout this whole section we assume that $A(0)\exa>1$.
The aim of this section is to show that 
starting from any point of $\dD\setminus\lbrace\,0\,\rbrace$,
the process $(Z_n)_{n\geq 0}$ creates a number of master sequences
of order $m$
with a reasonable probability,
within a time of order $\ln m$.
\begin{theorem}\label{enoughms}
For any $\g>0$ small enough,
there exists a positive constant $C$ 
such that for every $z\in\dD\setminus\lbrace\, 0\,\rbrace$
$$\liminf_\lmqq\,
\frac{1}{m}\ln P_z\big(
Z_{\lfloor C\ln m\rfloor}(0)\geq \g m
\big)\,=\,0\,.$$
\end{theorem}
Assume first that the process $(Z_t)_{t\geq0}$
starts from a neighborhood of one of the fixed points.
More precisely, let $b\in I_A\setminus \lbrace\,0\,\rbrace$ and
assume that the starting point is in a small neighborhood
of $\rho^b$ and is of the form
$$z\,=\,(w,0,\dots,0,z_b,\dots,z_K)\,.$$
Since, for $\d$ small enough, $G$ is contracting in the intersection
of the set $D_b=\lbrace\,r\in\cD:r_0+\cdots+r_{b-1}=0\,\rbrace$
with a sufficiently small neighborhood of $\rho^b$,
the process $(Z_t)_{t\geq0}$ will stay inside such a neighborhood
for a long time.
Note that for some $\e>0$ depending on the neighborhood,
$$G_0\Big(
\frac{z}{m}
\Big)\,\geq\,\frac{w A(0)\exa}{A(b)\exa+\e}\,.$$
A similar inequality holds for points close to $z$,
so that if the neighborhood is small enough,
as long as the process is inside it,
the number of master sequences will tend to increase geometrically.
This is the key idea of the proof of the theorem,
which will be carried out in a few different steps:

$\bullet$ First we show that from any starting point,
the process jumps to a point of the form of $z$
in a finite number of steps,
with probability higher than $e^{-\e m}$, for every $\e>0$.

$\bullet$ Then we build a deterministic trajectory that,
starting from a point of the form of $z$, creates $\g m$ master sequences 
in less than $C\ln m$ steps, with $\g, C>0$.

$\bullet$ Finally we show that 
the process is likely enough to follow the 
deterministic trajectory.

Before we begin with this strategy, let us give a few auxiliary results.
We define the mappings 
$\overline{F},\underline{F}:\cD\lra\cD$
by setting,
for $r\in\cD$ and $k\in\zk$
\begin{align*}
\underline{F}_k(r)\,&=\,\phi(r)^{-1}\sum_{0\leq i\leq k}
r_iA(i)M(i,k)\,,\\
\overline{F}_k(r)\,&=\,\phi(r)^{-1}\bigg(
\sum_{0\leq i\leq k}
r_iA(i)M(i,k)+\big(1-|\pi_k(r)|_1\big)A(0)M(k+1,k)\bigg)\,.
\end{align*}
The mappings $\underline{F}$ and $\overline{F}$ satisfy
$$\forall\,x\in\cS^\ell\quad
\forall\, k\in\zk\qquad
\underline{F}_k\big(
\pi(x)
\big)\,\leq\,
F_k(x)\,\leq\,
\overline{F}_k\big(
\pi(x)\big)\,.$$
The first inequality is always true,
while the second one holds asymptotically.
\begin{lemma}
\label{convunifGF}
The following limits are uniform in $r\in\cD$,
$$
\lim_\lq\,\big|
\underline{F}(r)-G(r)
\big|_1\,=\,0\,,\qquad
\lim_\lq\,\big|
\overline{F}(r)-G(r)
\big|_1\,=\,0\,.
$$
\end{lemma}
\begin{proof}
Recall that $M_\infty$
represents the limit mutation matrix (cf. the appendix~\ref{Propmut}).
Let $r\in\cD$ and $k\in\zk$, we have
\begin{multline*}
\big|
G_k(r)-\underline{F}_k(r)
\big|\,\leq\,\phi(r)^{-1}
\sum_{0\leq i\leq k}
r_iA(i)\big|M_\infty(i,k)-M(i,k)\big|\,\leq\\
\sup_{r\in\cD}\bigg(\phi(r)^{-1} 
\sum_{0\leq i\leq k} r_iA(i)
\bigg)
\max_{0\leq i,k\leq K}\big|
M_\infty(i,k)-M(i,k)
\big|\,.
\end{multline*}
This last quantity converges to $0$ asymptotically,
uniformly on $r\in\cD$. The rest of the lemma can be shown 
in a similar way.
\end{proof}
Results similar to propositions~\ref{fixed} and~\ref{convds} hold 
for the mapping $\underline{F}$.
The proofs are exactly the same,
even if the form of the fixed points is different.
More precisely,
we have the following result.
\begin{proposition}
Asymptotically, the mapping $\underline{F}$
has as many fixed points in $\cD$ as there are elements in $I_A$.
For each $b\in I_A$,
the associated fixed point $\eta^b$ satisfies $\eta^b_0=\cdots=\eta^b_{b-1}=0$
and $\eta^b_{b}\wedge\cdots\wedge\eta^b_K>0$.
The mapping $\underline{F}$ restricted to $\cD_b$ is
contracting in a neighborhood of $\eta^b$, and
$$\lim_\lq\eta^b\,=\,\rho^b\,.$$
\end{proposition}
The last convergence is a direct consequence of lemma~\ref{convunifGF}.
Let $\d>0$ and define, for $b\in I_A$, the sets
$$U_b(\d)\,=\,\big\lbrace\,
r\in\cD:|r-\eta^b|_1<\d
\,\big\rbrace\,.$$
Let $U(\d)$ denote the union of the sets $U_b(\d)$,
and let $W(\d)$ denote this same union but where
the neighborhood of $\eta^0$ has been left out,
i.e.,
$$
U(\d)\,=\,\bigcup_{b\in I_A} U_b(\d)\,,\qquad
W(\d)\,=\,U(\d)\setminus U_0(\d)\,.$$
Let $\e>0$.
The mapping $G$ is continuous on the compact set $\cD$,
it is therefore uniformly continuous on $\cD$.
In view of the lemma~\ref{convunifGF},
$\d$ can be chosen small enough so that 
for all $b\in I_A$, asymptotically,
$$\big|
\overline{F}(r)-\eta^b
\big|+\big|
\underline{F}(r)-\eta^b
\big|<\e\,.$$
Let $z\in\dD\setminus\lbrace\,0\,\rbrace$,
and note that $G(z)\neq 0$.
By lemma~\ref{timefar1}, 
there exist $h\in\N$ and $c>0$ such that, asymptotically,
$$P_z\big(
\exists\, i\in\lbrace\,0,\dots,h\,\rbrace\ \text{such that}\
Z_i\in\dU\setminus\lbrace\,0
\big)\,\geq\,1-e^{-cm}\,.$$
Suppose now that $z\in\dU_b(\d)$ for some $b\in I_A$,
fix $w\in\N$ and set
$z'=(w,0,\dots,0,z_b,\dots,z_K)$.
We show next that,
for $\d$ small enough,
$$\lim_\lmqq\,\frac{1}{m}\ln
P_{z}\big(
Z_1=z'
\big)\,\geq\,-\e\,.$$
Indeed, note that for any $o\in\pml$ satisfying $\pi(o)=z$,
we have the following asymptotic bound on the probability 
of creating a master sequence,
$$
F_0\Big(
\frac{o}{m}
\Big)\,\geq\,\phi\Big(
\frac{z}{m}
\Big)^{-1}
\sum_{0\leq h\leq K}\frac{z_h}{m}A(h)M(h,0)
\,\geq\,m^{-M}\,,
$$
for some $M>0$.
Which implies the following lower bound
on the probability of jumping from $z$ to $z'$,
$$
\frac{m!}{w!z_b!\cdots z_K!(m-|z'|_1)!}
m^{-Mw}\prod_{i=b}^K
\underline{F}_i\Big(
\frac{z}{m}\Big)^{z_i}\times
\bigg(
1-\Big|
\overline{F}\Big(
\frac{z}{m}
\Big)
\Big|_1
\bigg)^{m-|z'|_1}\,.
$$
We use now the lemma~\ref{ldmultinom}
in order to obtain the following asymptotic bound,
\begin{multline*}
\frac{1}{m}\ln 
P\big(
Z_1=z'\,\big|\,Z_0=z
\big)\,\geq\,
-\frac{1}{m}(K+1)(2+\ln m)
-\frac{w}{m}\ln\frac{w/m}{m^{-M}}\\
-\sum_{i=b}^K\frac{z_i}{m}\ln\frac{z_i/m}{\underline{F}_i(z/m)}
-\frac{m-|z'|_1}{m}\ln\frac{(m-|z'|_1)/m}{1-|\overline{F}(z/m)|_1}
\,.\end{multline*}
The first two quantities go to $0$, when $m$ goes to infinity,
so that the sum of both is eventually larger that $-\e/2$.
Since $z/m\in U_b(\d)$,
we have
$$\Big|
\frac{z}{m}-\underline{F}\Big(
\frac{z}{m}
\Big)
\Big|_1\,\leq\,
\Big|
\frac{z}{m}-\eta^b
\Big|+\Big|
\eta^b-\underline{F}\Big(
\frac{z}{m}
\Big)
\Big|\,<\,\d+\e\,.$$
Furthermore, for $b\leq i\leq K$,
the function $\underline{F}_i(r)$ is bounded
below by a positive constant $c$ when $r\in U_b(\d)$,
we conclude that
$$\sum_{i=b}^K
\frac{z_i}{m}\ln\frac{z_i/m}{\underline{F}_i(z/m)}\,\leq\,
K\ln\Big(
1+\frac{\d+\e}{c}
\Big)\,.$$
We choose $\d$ small enough so that this last quantity is
smaller than $\e/4$.
A similar argument shows that 
$\d$ can be chosen small enough
so that the last term is also bounded below by $-\e/4$, 
thus giving the desired bound;
$$\lim_\lmqq\,\frac{1}{m}\ln
P_{z}\big(
Z_1=z'
\big)\,\geq\,-\e\,.$$
Let us set $\d>0$, $b\in I_A\setminus\lbrace\,0\,\rbrace$ and $w\in\N$.
Suppose that 
$$z\,=\,(w,0,\dots,0,z_b,\dots,z_K)\,\in\,\dU_{b}(\d/2)\,.$$
We build next a deterministic trajectory $(z^n)_{n\geq0}$
such that $z^0=z$ and for $\d$ small enough, $w$ large enough,
there exist $\eta,C>0$ and $N< C\ln m$ satisfying
$z^N_0\geq \eta m$.
We set $z^0=z$ and for $n\geq1$,
$$
z^n\,=\,\bigg\lfloor
m\underline{F}\Big(
\frac{z^{n-1}}{m}
\Big)
\bigg\rfloor\,.
$$
Denote by $\overline{\phi}$ and $\underline{\phi}$
the maximum and the minimum mean fitness of a population
in $U_b(\d)$, i.e.,
$$\overline{\phi}\,=\,\max_{r\in U_b(\d)}\phi(r)\,,\qquad
\underline{\phi}\,=\,\min_{r\in U_b(\d)}\phi(r)\,.$$
We define by $N_b(\d)$ the first time of exit of the dynamical 
system $z^n$ from $\dU_b(\d)$:
$$N_b(\d)\,=\,\inf\big\lbrace\,
n\geq0:z^n\not\in \dU_b(\d) 
\,\big\rbrace\,.$$
Take $\d$ small enough and $w$ large enough so that, asymptotically,
$$\rho\,=\,\frac{A(0)M(0,0)}{\overline{\phi}}-\frac{1}{w}\,>1,.$$
Then, asymptotically, for any $n\leq N_b(\d)$,
$$z^n_0\,=\,\bigg\lfloor
\frac{z^{n-1}_0A(0)M(0,0)}{\phi(z^0/m)}
\bigg\rfloor\,\geq\,z^{n-1}_0\bigg(
\frac{A(0)M(0,0)}{\overline{\phi}}-\frac{1}{w}
\bigg)\,=\,\rho z^{n-1}_0\,.$$
The sequence
$(z^0_n)_{0\leq n\leq N_b(\d)}$ is increasing,
and bounded below by a geometric sequence with ratio $\rho$.
Let $\g>0$, then,
$\rho^n w$ is larger than $\g m$ if
$$n\,\geq\,n(\g)\,=\,\big(
\ln\rho
\big)^{-1}\ln\frac{\g m}{w}\,.$$
Therefore,
we try to show next that there exists $\g>0$ such that 
$N_b(\d)$ is larger than the quantity $n(\g)$.
In order to do so, we show in the next lemma that the coordinates $(z^n_k)_{0\leq k<b}$
cannot grow at a faster rate than $z^n_0$.
We prove afterwards that the same thing holds for the difference $|z_k^n-\eta^b_k|$,
where $b\leq k\leq K$.
Let us choose $\e>0$ small enough so that, asymptotically,
$$\forall\,k\in\lbrace\,1,\dots,K\,\rbrace\qquad
\frac{A(k)M(k,k)}{A(0)M(0,0)-\e}\,<\,1\,,$$
and let $w$ be such that $\overline{\phi}/w<\e$.
We prove next the following lemma. 
\begin{lemma}
There exist positive constants $c_0,\dots,c_{b-1}$
such that for $0\leq k<b$,
asymptotically,
$z_k^n\leq c_kz_0^n\,,$
for all $n\leq N_b(\d)$.
\end{lemma}
\begin{proof}
We will prove the lemma by induction on $k$.
The case $k=0$ is obviously true.
Set $k\in\lbrace\,1,\dots,b-1\,\rbrace$
and suppose that the statement of the theorem holds for 
the coordinates $0,\dots,k-1$. Then, we have
$$\frac{z^n_k}{z^n_0}\,\leq\,
\frac{\displaystyle\phi(z^{n-1}/m)^{-1}
\sum_{0\leq i\leq k}z^{n-1}_iA(i)M(i,k)}{\displaystyle\phi(z^{n-1}/m)^{-1}z^{n-1}_0A(0)M(0,0)-1}\,
\leq\sum_{0\leq i\leq k}\frac{z^{n-1}_iA(i)M(i,k)}
{z^{n-1}_0A(0)M(0,0)-\overline{\phi}}\,.
$$
The sequence $(z^n_0)_{0\leq n\leq N_b(\d)}$ is increasing,
and thus, $\overline{\phi}/z^{n-1}_0<\overline{\phi}/w<\e$.
Therefore,
$$\frac{z^n_k}{z^n_0}\,\leq\,
\sum_{0\leq i\leq k}\frac{z^{n-1}_i}{z^{n-1}_0}\bigg(
\frac{A(i)M(i,k)}{A(0)M(0,0)-\e}
\bigg)\,.$$
By the induction hypothesis
$$\frac{z^n_k}{z^n_0}\,\leq\,
\sum_{0\leq i\leq k-1}c_i\bigg(
\frac{A(i)M(i,k)}{A(0)M(0,0)-\e}
\bigg)
+\frac{z^{n-1}_k}{z^{n-1}_0}\frac{A(k)M(k,k)}{A(0)M(0,0)-\e}\,.$$
Iterating this inequality,
and noting that $z^0_k=0$, we obtain
$$\frac{z^n_k}{z^n_0}\,\leq\,
\sum_{0\leq i\leq k-1}c_i\bigg(
\frac{A(i)M(i,k)}{A(0)M(0,0)-\e}
\bigg)\sum_{t=0}^{n-1}
\bigg(\frac{A(k)M(k,k)}{A(0)M(0,0)-\e}\bigg)^t\,.$$
Yet, asymptotically,
$$d_k\,=\,\frac{A(k)M(k,k)}{A(0)M(0,0)-\e}\,<\,1\,,$$
and thus,
$$\frac{z^n_k}{z^n_0}\,\leq\,\frac{1}{1-d_k}\sum_{0\leq i\leq k-1}c_i\bigg(
\frac{A(i)M(i,k)}{A(0)M(0,0)-\e}\bigg)\,.$$
We take $c_k$ to be equal to the right hand side
of this inequality, which fulfills the proof of the lemma.
\end{proof}
Let us introduce the following notation, for any $x\in\R^{K+1}$
we denote by $\widetilde{x}$ the vector $x$ where the 
coordinates $0,\dots,b-1$ have been set to $0$, i.e.,
$$\widetilde{x}\,=\,(0,\dots,0,x_b,\dots,x_K)\,.$$
\begin{lemma}
There exist constants $c_b>0$
and $0<c_\d<1$
such that
for $n\leq N_b(\d)$, 
asymptotically,
we have
$$\big|\widetilde{z^{n}}-m\eta^b\big|_1
\,\leq\,
c_b z^n_0+c_\d^n\big|
\widetilde{z^0}-m\eta^b
\big|_1\,.
$$
\end{lemma}
\begin{proof}
For $n\geq 0$ and $b\leq k\leq K$, we have,
noting that $\eta^b$ is a fixed point of the mapping $\underline{F}$,
$$\big|
z^n_k-m\eta^b_k
\big|\,\leq\,1+m\big|
\underline{F}_k(z^{n-1}/m)-\underline{F}_k(\eta^b)
\big|\,.$$
Yet, for all $r\in\cD$
$$\underline{F}_k(r)\,=\,\sum_{i=0}^{b-1}r_i\frac{A(i)M(i,k)}{\phi(r)}
+\underline{F}_k(\widetilde{r})\frac{\phi(\widetilde{r})}{\phi(r)}\,.$$
Thus,
\begin{multline*}\big|
z^n_k-m\eta^b_k
\big|\,\leq\,1+\sum_{i=0}^{b-1}z^{n-1}_i\frac{A(i)M(i,k)}{\phi(z^{n-1}/m)}+m\big|
\underline{F}_k(\widetilde{z}^{n-1}/m)-\underline{F}_k(\eta^b)
\big|\\+\underline{F}_k(\widetilde{z}^{n-1}/m)\frac{|\phi(\widetilde{z^{n-1}}/m)-\phi(z^{n-1}/m)|}{\phi(z^{n-1}/m)}\,.
\end{multline*}
We have,
$$\big|
\phi(\widetilde{z^{n-1}}/m)-\phi(z^{n-1}/m)
\big|\,\leq\,\sum_{j=0}^{b-1}\frac{z_j^{n-1}}{m}|A(j)-1|$$
Reporting back in the inequality for $|z^n_k-m\eta^b_k|$, we get,
\begin{multline*}
\big|
z^n_k-m\eta^b_k
\big|\,\leq\,1+m\bigg|
\underline{F}_k(
\widetilde{z^{n-1}}/m)-\underline{F}_k(\eta^b)
\bigg|\\
+\sum_{j=0}^{b-1}z^{n-1}_j\Bigg(
\frac{A(j)M(j,k)+|A(j)-1|\underline{F}_k(\widetilde{z^{n-1}}/m)}{\phi(z^{n-1}/m)}
\Bigg)
\,.
\end{multline*}
Summing from $k=b$ to $K$,
and recalling that, asymptotically, $\underline{F}$ is contracting on the set $U_b(\d)\cap\cD_b$, 
we deduce the existece of a constant $c_\d<1$ such that
$$
\big|
\widetilde{z^n}-m\eta^b
\big|_1\,\leq\,
K+c_\d\big|
\widetilde{z^{n-1}}-m\eta^b
\big|_1+
\sum_{j=0}^{b-1}z_j^{n-1}\frac{A(j)\displaystyle\sum_{b\leq k\leq K}M(j,k)+|A(j)-1|}{\overline{\phi}}
\,,
$$
Using the previous lemma, we get
that, asymptotically,
$$\big|
\widetilde{z^n}-m\eta^b
\big|_1\,\leq\,Cz^n_0+c_\d\big|
\widetilde{z^{n-1}}-m\eta^b
\big|_1\,,$$
for some constant $C$ depending on $\d$ only.
Iterating this inequality,
and noting that for $t\leq n$, we have $z^{n-t}_0\leq \rho^{-t}z_0^n$,
we conclude that
\begin{multline*}
\big|
\widetilde{z^n}-m\eta^b
\big|_1\,\leq\,C\sum_{t=0}^{n-1}c_\d^tz_0^{n-t}+c_\d^n\big|
\widetilde{z^0}-m\eta^b
\big|_1\\
\leq\,Cz_0^n\sum_{t=0}^{n-1}(c_\d/\rho)^t+c_\d^n\big|
\widetilde{z^0}-m\eta^b
\big|_1\,\leq\,
C\big(
1-c_\d/\rho
\big)^{-1}z_0^n+c_\d^n\big|
\widetilde{z^0}-m\eta^b
\big|_1\,.
\end{multline*}
The proof is achieved by taking $c_b=C(1-c_\d/\rho)^{-1}$.
\end{proof}
As a consequence of these two lemmas,
we have 
$$z^n_0\geq
\frac{z^n_1}{c_1}\vee\cdots\vee\frac{z^n_{b-1}}{c_{b-1}}\vee
\frac{|\widetilde{z^n}-m\eta^b|_1-|\widetilde{z^0}-m\eta^b|_1}{c_b}\,.$$
Thus, taking $\g(K+1)<\min\,\lbrace\,\d/c_1,\dots,\d/2c_b\,\rbrace$,
then $N_b(\d)\geq n(\g)$, as wanted.
We show next that the process $(Z_n)_{n\geq0}$
has a fairly high probability to follow the deterministic
trajectory that we have just built.
\begin{lemma}
Let $z^0=(w,0,\dots,z_b,\dots,z_K)\in\dU_b(\d/2)$
and let $(z^n)_{n\geq0}$ be the trajectory built from $z^0$ 
by setting $z^n=\lfloor m\underline{F}(z^{n-1}/m)\rfloor $,
and let $\g$ be as above.
 We have,
 $$\liminf_\lmqq\,\frac{1}{m}\ln P_{z^0}\big(
Z_{n(\g)}=z^{n(\g)},\dots,Z_1=z^1\big)\,=\,0\,.$$
 \end{lemma}
\begin{proof}
We have,
$$
P_{z^0}\big(
Z_{n(\g)}=z^{n(\g)},\dots,Z_1=z^1
\big)\,=\,
\prod_{n=0}^{n(\g)-1}P_{z^0}\big(
Z_{n+1}=z^{n+1}\,\big|\,Z_n=z^n
\big)\,.
$$
For any $0\leq n\leq n(\g)$, 
as in the proof of the large deviations
principle~\ref{pgdtrans1}, 
$$P_{z^0}\big(
Z_{n+1}=z^{n+1}\,\big|\,Z_n=z^n
\big)\,\geq\,\inf_{o:\pi(0)=z^n}
P\big(
Z_{n+1}=z^{n+1}\,\big|\,O_n=o
\big)\,.$$
Let $o\in\pml$
be such that $\pi(o)=z^n$. We have,
\begin{multline*}
P\big(
Z_{n+1}=z^{n+1}\,\big|\,O_n=o
\big)\,=\,
\frac{m!}{z^{n+1}_0!\cdots z^{n+1}_K!
(m-|z^{n+1}|_1)!}\\
\times
F_0\Big(
\frac{o}{m}
\Big)^{z^{n+1}_0}\cdots
F_K\Big(
\frac{o}{m}
\Big)^{z^{n+1}_K}\bigg(
1-\bigg|
\pi\Big(
F \Big(
\frac{o}{m}
\Big)
\Big)
\bigg|
\bigg)^{m-|z^{n+1}|_1}\,.
\end{multline*}
Thus,
$$
\ln P\big(
Z_{n+1}=z^{n+1}\,\big|\,O_n=o
\big)\,=\,-mI_K\bigg(
\pi\Big(
F\Big(
\frac{o}{m}
\Big)
\Big),\frac{z^{n+1}}{m}
\bigg)+\Phi(o,z^{n+1})\,,
$$
where the error term $\Phi(o,z^{n+1})$
satisfies
$$
\big|\Phi(o,z^{n+1})\big|\,\leq\, C(K)(\ln m+1)\,,$$
$C(K)$ being a constant that depends on $K$ but not on $m$.
Next we bound the quantity involving
the rate function $I$.
Recall that 
\begin{multline*}
mI_K\bigg(
\pi\Big(
F\Big(
\frac{o}{m}
\Big)
\Big),\frac{z^{n+1}}{m}
\bigg)\\=\,
\sum_{k=0}^K
z^{n+1}_k\ln\frac{z^{n+1}_k/m}{F_k(o/m)}+
\big(
m-|z^{n+1}|_1
\big)\ln\frac{1-|z^{n+1}|_1/m}
{1-|\pi(F(o/m))|_1}\,.
\end{multline*}
The function $\underline{F}$
has been defined so that for all $x\in\cS^\ell$,
${\underline{F}_k\big(\pi(x)\big)\leq F_k(x)}$, for $0\leq k\leq K$. Therefore,
for all $o\in\pml$ such that $\pi(o)=z^n$,
$$
\frac{z^{n+1}_k}{m}\,=\,
\frac{1}{m}\bigg\lfloor
m\underline{F}_k\Big(
\frac{z^{n}}{m}
\Big)
\bigg\rfloor\,\leq\,\underline{F}_k\Big(
\frac{z^{n}}{m}
\Big)\,\leq\,
F_k\Big(
\frac{o}{m}
\Big)\,.$$
Thus, 
$$mI_K\bigg(
\pi_k\Big(
F\Big(
\frac{o}{m}
\Big)
\Big),\frac{z^{n+1}}{m}
\bigg)\,\leq\,\big(
m-|z^{n+1}|_1
\big)\ln\frac{1-|z^{n+1}|_1/m}
{1-|\pi(F(o/m))|_1}\,.$$
The argument of the logarithm is larger than 1,
and for all $x\geq0$, $\ln(x)\leq x-1$. Therefore,
the above quantity is bounded by
\begin{multline*}
\big(
m-|z^{n+1}|_1
\big)
\frac{\displaystyle\big|\pi\big(F(o/m)\big)\big|_1-|z^{n+1}|_1/m}
{1-\displaystyle\big|\pi\big(F(o/m)\big)\big|_1}\\
\leq\,
\frac{m-|z^{n+1}|_1}{1-\displaystyle\big|\pi\big(F(o/m)\big)\big|_1}
\sum_{k=0}^K\bigg|
\frac{1}{m}\bigg\lfloor
m\underline{F}_k\Big(
\frac{z^{n}}{m}
\Big)
\bigg\rfloor-
F_k\Big(
\frac{o}{m}
\Big)
\bigg|\\
\leq\,\frac{m-|z^{n+1}|_1}{1-\displaystyle\big|\pi\big(F(o/m)\big)\big|_1}
\sum_{k=0}^K\Bigg(
\bigg|
\underline{F}_k\Big(
\frac{z^{n}}{m}
\Big)-
F_k\Big(
\frac{o}{m}
\Big)
\bigg|+\frac{1}{m}\Bigg)\,.
\end{multline*}
For any $x\in\cS^\ell$
$$\big|
\pi\big(
F(x)
\big)
\big|_1\,=\,
\phi\big(\pi(x)\big)^{-1}\sum_{h=0}^\ell
x_hA(h)\sum_{k=0}^KM(h,k)\,.$$
On one hand, for any $h\in\zl$ 
the sum $M(h,0)+\cdots+M(h,K)$ 
is bounded by a constant $c$ which is strictly 
smaller than 1. Thus,
$|\pi_K(F(x))|_1$ is bounded above by
this same constant $c$, uniformly on $x\in\cS^\ell$.
On the other hand,
for $0\leq k\leq K$, since $\pi(o)=z^n$,
$$
\bigg|
\underline{F}_k\Big(
\frac{z^{n}}{m}
\Big)-
F_k\Big(
\frac{o}{m}
\Big)
\bigg|\,=\,
\phi(z^n/m)^{-1}
\sum_{h=k+1}^\ell\frac{o_h}{m}A(h)M(h,k)\,.
$$
Yet,
there exists a positive constant $c'$
such that asymptotically, $M(h,k)\leq c'/m$,
for $0\leq k<h\leq \ell$.
Therefore, the above quantity
is bounded by $c'/m$.
We conclude that
$$I_K\bigg(
\pi\Big(
F\Big(
\frac{o}{m}
\Big)
\Big),\frac{z^{n+1}}{m}
\bigg)\,\leq\,\frac{(1+c')(K+1)}{m(1-c)}\,.$$
Therefore,
\begin{multline*}
\frac{1}{m}\ln P_{z^0}\big(
Z_{n(\g)}=z^{n(\g)},\dots,Z_1=z^1
\big)\\\geq\,
-\sum_{n=0}^{n(\g)-1}
\inf_{o:\pi(o)=z^n}
\bigg(I_K\bigg(
\pi\Big(
F\Big(
\frac{o}{m}
\Big)
\Big),\frac{z^{n+1}}{m}
\bigg)+\frac{1}{m}\Phi(o,z^{n+1})\bigg)\\
\geq\,-\frac{n(\g)(1+c')(K+1)}{m(1-c)
}-\frac{n(\g)C(K)(\ln m+1)}{m}\,.
\end{multline*}
Since $n(\g)$ is of the order of $\ln m$,
we see that this last quantity goes to 0 when $m$ goes to infinity,
as wanted.
\end{proof}
Combining the previous lemmas 
we obtain the proof of theorem~\ref{enoughms}.
Once the process $Z_n$ has $\g m$ master sequences,
it will converge to $\rho^0$ in a few steps.
Indeed, let $\d>0$ and define
$$U_\d\,=\,\big\lbrace\,
r\in\cD:|r-\rho^0|_1<\d
\,\big\rbrace\,.$$
We have the following corollary.
\begin{corollary}\label{enterneigh}
Let $\d>0$.
There exists a positive constant $C$
such that for every $z\in\dD\setminus\lbrace\, 0\,\rbrace$
$$\liminf_\lmq\,
\frac{1}{m}\ln P_z\big(
Z_{\lfloor C\ln m\rfloor}(0)\in\dU_\d
\big)\,=\,0\,.$$
\end{corollary}
\begin{proof}
Let $\g>0$ be small enough and let $C'$ associated to $\g$ as in theorem~\ref{enoughms}.
Then, for any $C>C'$ and $z\in\dD\setminus\lbrace\,0\,\rbrace$,
we have,
$$
P_z\big(
Z_{\lfloor C\ln m\rfloor}\in\dU_\d
\big)\geq\!\!\!\!\!\!\sum_{z'\in\dD:z'(0)\geq \g m}\!\!\!\!\!\!P_z\big(
Z_{\lfloor C'\ln m\rfloor}=z'
\big)P_z\big(
Z_{\lfloor C\ln m\rfloor}\in\dU_\d\big|
Z_{\lfloor C' \ln m\rfloor=z'}
\big).
$$
By lemma~\ref{timefar1},
there exist $h\in\N$ and $c>0$ such that,
asymptotically, for every $z'\in\dD$ satisfying $z_0'\geq \eta m$
$$P_{z'}\big(
Z_h\in\dU_\d
\big)\,\geq\,1-e^{-c m}\,.$$
Thus, taking $C$ such that $\lfloor C\ln m\rfloor-\lfloor C'\ln m\rfloor>h$,
and in view of theorem~\ref{enoughms}, we obtain the result of the corollary.
\end{proof}

\section{Persistence time}\label{Pertime}
We assume throughout this whole section that $A(0)\exa>1$.
The aim of this section is to compute the expected hitting time of $0$
for the process $(Z_n)_{n\geq0}$.
In order to ease the readability of the upcoming formulas,
we denote the quantity $V(\rho^0,0)$ simply as $V$.
Let us define 
$$\tau_0\,=\,\inf\big\lbrace\,n\geq0:Z_n=0\,\big\rbrace\,.$$
We will prove the following result.
\begin{theorem} For all $z\in\dD$,
$$\lim_{\genfrac{}{}{0pt}{1}{\ell,m\to\infty,\,q\to0}{\ell q\to a,\,m/\ell\to\a}}\,
\frac{1}{m}\ln E_z(\tau_0)\,=\,V\,.$$
\end{theorem}
\begin{proof}
We begin by showing the upper bound: 
for any $\e>0$, we have 
$$\forall\,z\in\dD\qquad\limsup_{\genfrac{}{}{0pt}{1}{\ell,m\to\infty,\,q\to0}{\ell q\to a,\,m/\ell\to\a}}\,
\frac{1}{m}\ln E_z(\tau_0)\,\leq\,V+\e\,.$$
Let $\e>0$.
We first show that there exists a constant 
$C>0$
such that 
$$\forall\,z\in\dD\setminus\lbrace\,0\,\rbrace\qquad
P_z\big(\tau_0\leq \lfloor C\ln m\rfloor\big)
\,\geq\,e^{-m(V+2\e)}$$
Let $\g>0$, $z\in\dD\setminus\lbrace\,0\,\rbrace$,
and assume first that $z_0>\g m$.
Define the sequence $(r^n)_{n\geq0}$ by setting
$r^0=z/m$ and
$$r^n\,=\,G^n(r^0)\,,\qquad n\geq1\,.$$
The mapping $V_1$ is continuous on the first argument
in a neighborhood of $\rho^0$;
let us choose $\d$ small enough so that 
$$|r-\rho^0|_1\,<\,\d\quad\Rightarrow\quad
V_1(r,\rho^0)\,<\,\e/3\,.$$
Moreover, for $\d$ small enough
there exists $h\in\N$ such that for all $r\in\cD$ satisfying $r\geq \g$,
and for all $n\geq h$, we have
$$\big|G^n(r)-\rho^0\big|_1\,<\,\d\,.$$
Indeed, by the theorem~\ref{convds}, $\d$ can be chosen sufficiently small so 
that the {$\d$--neighborhood} of $\rho^0$ is contracting.
By continuity of the map $G$,
for all $r\in\cD$ such that $r_0>\g$,
there exists $\d_r>0$ and $h(r)\in\N$ such that if 
$$|r-t|<\d_r\quad\Lra\quad \big|G^{h(r)}(t)-\rho^0\big|\,<\,\d\,.$$
The set $\dD_\g=\lbrace\,r\in\cD:r_0\geq \g\,\rbrace$ is compact and the family
$\lbrace\,B(r,\d_r):r\in\dD_\g\,\rbrace$
is an open cover of the set $\dD_\g$.
Thus,
there exist $r_1,\dots,r_N\in\dD_\g$
such that
$$\dD_\g\,\subset\,\bigcup_{i=1}^N B(r_i,\d_{r_i})\,.$$
Set
$h$ to be the maximum of $h(r_1),\dots,h(r_n)$.
Then, for all $r\in\dD_\g$ and $n\geq h$,
we have $|G^n(r)-\rho^0|<\d$.
Let $h'\geq 0$ and let $(t^i)_{0\leq i\leq h'}$
be a sequence in $\cD$ satisfying
$$t^0\,=\,\rho^0\,,\ t^{h'}\,=\,0\,,\qquad
\sum_{i=0}^{h'-1}V_1(t^i,t^{i+1})\,\leq\,V(\rho^0,0)+\frac{\e}{3}\,.$$
Consider next the sequence $(s_i)_{0\leq i\leq h+h'+1}$
defined by
\begin{align*}
s_0&=r^0\,,\quad s_1=r^1\,,\quad\dots\,,\quad s_h=r^h\,,\\
s_{h+1}&=t^0=\rho^0\,,\quad s_{h+2}=t^1\,,\quad\dots\,,\quad s_{h+h'+1}=t^{h'}=0\,.
\end{align*}
Set $l=h+h'+1$. The sequence $(s_i)_{0\leq i\leq l}$ satisfies 
$$\sum_{i=0}^{l-1}V_1(s_i,s_{i+1})\,\leq\,V(\rho^0)+2\e/3\,.$$
Proceeding as in the proof of the large deviations principle~\ref{pgdtrans1},
we see that
$$P_z\big(
Z_{l}=0
\big)\,\geq\,
\prod_{t=0}^{l-1}P_z\big(
Z_{t+1}=\lfloor s_{t+1}m\rfloor\,\big|\,
Z_t=\lfloor s_tm\rfloor\big)\,.
$$
Then,
$$\liminf_{\genfrac{}{}{0pt}{1}{\ell,m\to\infty,\,q\to 0}{\ell q\to a,\,m/\ell\to\a}}\frac{1}{m}\,\ln
P_{z}\big(
Z_{l}=0
\big)\,\geq\,-V+2\e/3\,.$$
Thus, asymptotically,
$$P_{z}\big(
Z_{l}=0
\big)\,\geq\,e^{-m(V+\e)}\,,$$
uniformly on $z\in\dD_\g$.
Suppose now that $z\not\in\dD_\g$.
By theorem~\ref{enoughms}, there exist
$C'>0$ such that
$$\forall\, z\in\dD\setminus\lbrace\,0\,\rbrace\qquad
P_z\big(
Z_{\lfloor C'\ln m\rfloor}\geq \g m
\big)\,\geq e^{-\e m}\,.$$
Thus, for every $z\in\dD\setminus\lbrace\,0\,\rbrace$,
\begin{multline*}
P_z\big(
Z_{\lfloor C'\ln m\rfloor+l}=0
\big)\,
\geq\,
\sum_{z'\in\dD_\g}P_z\big(
Z_{\lfloor C'\ln m\rfloor+l}=0, Z_{\lfloor C'\ln m\rfloor}=z'
\big)\\
\geq\,
\sum_{z'\in\dD_\g}
P_{z'}\big(
Z_l=0
\big)P_z\big(
Z_{N'}=z'
\big)\,\geq\,e^{-m(V+2\e)}\,.
\end{multline*}
Taking $C$ such that $\lfloor C\ln m\rfloor\geq \lfloor C'\ln m\rfloor+l$, we conclude that for every $z\in\dD$,
$$P_z\big(
\tau_0\leq\lfloor C\ln m\rfloor \big)\,\geq\,
e^{-m(V+2\e)}\,.$$
Proceeding as in corollary~\ref{timefar2} we obtain that, for every $h\geq1$ and $z\in\dD\setminus\lbrace\,0\,\rbrace$,
$$
P_z\big(
\tau_0\geq h \lfloor C\ln m\rfloor
\big)\,\leq\,\big(
1-e^{-m(V+2\e)}
\big)^h\,.$$
Thus,
\begin{multline*}
E_z(\tau_0)\,=\,\sum_{n\geq0}P_z\big(
\tau_0\geq n
\big)\,
=\,
\sum_{h\geq 0}\sum_{n=h\lfloor C\ln m\rfloor}^{(h+1)\lfloor C\ln m\rfloor-1}
P_z\big(
\tau_0\geq n
\big)\\
\leq\,\sum_{h\geq0}\lfloor C\ln m\rfloor P_z\big(
\tau_0\geq h \lfloor C\ln m\rfloor
\big)\,
\leq\,
\lfloor C\ln m\rfloor e^{m(V+2\e)}\,.
\end{multline*}
We conclude that, for every $z\in\dD$,
$$\limsup_{\genfrac{}{}{0pt}{1}{\ell,m\to\infty,\,q\to 0}{\ell q \to a,\,m/\ell\to\a}}\,\frac{1}{m}\ln E_z(
\tau_0)\,\leq\,V+2\e\,,$$
We send $\e$ to 0 and we obtain the desired
upper bound.
We proceed now to the proof of the lower bound.
Let us set $\d>0$ and let us define,
for $x\in\cD$ the $\d$--neighborhood of $x$ by 
$$U_\d(x)\,=\,\big\lbrace\,
y\in\cD:|x-y|_1<\d
\,\big\rbrace\,.$$
Equally,
we set 
$\dU_\d(x)=mU_d(x)\cap \dD$.
Let $\tau_\d$ be the hitting time of
the $\d$--neighborhood of $0$,
$$\tau_\d\,=\,\inf\big\lbrace\,
n\geq 0:Z_n\in\dU_\d(0)
\,\big\rbrace\,.$$
Obviously, $\tau_\d\leq \tau_0$.
We will first show that 
for every $z\in\dU_\d(\rho^0)$,
$$\liminf_{\genfrac{}{}{0pt}{1}{\ell,m\to\infty,\,q\to0}{\ell q\to a,\,m/\ell\to\a}}\,\frac{1}{m}\ln E_z(\tau_\d)
\,\geq\,\inf\big\lbrace\,
V(r,t):r\in U_\d(\rho^0),t\in U_\d(0)
\,\big\rbrace-\d\,.$$
In order to ease the notation in the sequel,
we set $V^\d$ to be the infimum appearing in the above formula, and we write $P_z,E_z$
for the probabilities and expectations associated with the process $(Z_n)_{n\geq0}$
starting from $z$.
Using Markov's inequality,
for all $T\geq 0$
$$E_z(\tau_\d)\,\geq\,
TP_z(\tau_\d\geq T)\,.$$
Thus, we set $T\geq 0$ and we bound
the probability of the event $\lbrace\,\tau_\d<T\,\rbrace$.
Let us denote by $T_0$ the last time before $\tau_\d$ 
that the process is in $\dU_\d(\rho^0)$, i.e.,
$$T_0\,=\,\max\big\lbrace\,
n\leq \tau_\d:Z_n\in\dU_\d(\rho^0)
\,\big\rbrace\,.$$
We will bound the probability of the event $\lbrace\,\tau_\d<T\,\rbrace$
by studying the trajectory of the process $\Zt$ between $T_0$ and $\tau_\d$.
The idea is the following,
either the trajectory $(Z_n)_{T_0<n<\tau_\d}$
spends a long time outside a neighborhood of $\rho^0$ and $0$,
which is very unlikely (lemmas~\ref{timefar1}, \ref{enoughms} and corollary~\ref{timefar2}),
or it jumps in a few steps from one fixed point to another until reaching 0,
in which case the lower bound of the large deviations principle~\ref{pgdtrans1}
will give us the desired estimate.
We prove first a useful lemma.
\begin{lemma}\label{timefar0rho}
Let $\e>0$. For all $T>e^{2\e m}$
and for all $z\in\dD\setminus\lbrace\,0\,\rbrace$,
we have, asymptotically,
$$P_z\big(
Z_t\not\in\dU_\d(\rho^0)\cup\dU_\d(0),0\leq t\leq T
\big)\,\leq\,e^{-\e m}\,.$$
\end{lemma}
\begin{proof}
Let $\e,\g>0$,
as we have shown in the proof of the upper bound,
there exists $h\in\N$ such that for every $z\in\dD$ satisfying $z_0\geq \g m$,
$$P_z\big(
Z_h\in\dU_\d(\rho^0)\cup \dU_\d(0)\big)\,\geq\,
P_z\big(
Z_h\in\dU_\d(\rho^0)\big)\,\geq\,e^{-\e m /2}
\,.$$
In view of theorem~\ref{enoughms},
there exists $C>0$ such that for every 
$z\in \dD\setminus\lbrace\,0\,\rbrace$,
$$P_z\big(
Z_{\lfloor C\ln m\rfloor}(0)\geq \g m
\big)\,>\,e^{-\e m/2}\,.
$$
Thus, taking $l=\lfloor C\ln m\rfloor+h$,
we have, for every
$z\in \dD\setminus\lbrace\,0\,\rbrace$,
\begin{multline*}P_z\big(
Z_l\in \dU_\d(\rho^0)\cup\dU_\d(0)
\big)\\
\geq\,\sum_{z'\in\dD:z'_0\geq\g m}
P_z\big(Z_{\lfloor C\ln m\rfloor}=z'\big)P_{z'}\big(Z_h\in\dU_\d(\rho^0)\cup\dU_\d(0)\big)\\
\geq\sum_{z'\in\dD:z'_0\geq\g m}P_z\big(
Z_{\lfloor C\ln m\rfloor}=z'
\big)e^{-\e m/2}\,>\,e^{-\e m}\,.
\end{multline*}
Proceeding as in corollary~\ref{timefar2},
for every $j\in\N$ and for every $z\in\dD\setminus(\dU_\d(\rho^0)\cup\dU_\d(0))$,
$$P_z\big(
Z_n\not\in\dU_\d(\rho^0)\cup\dU_\d(0), 0\leq n\leq jl
\big)\,\leq\,(1-e^{-\e m})^j\,.$$
Taking $j>e^{2\e m}/l$,
this probability is smaller than $e^{-\e m}$, as wanted.
\end{proof}
We continue now with the proof of the lower bound.
We have, for $T\geq0$, $z\in\dU_\d(\rho^0)$, and $k\leq T$,
\begin{multline*}
P_z(\tau_\d< T)\,=\,\sum_{0\leq t_0<t^*< T}P_z(T_0=t_0,\tau_\d=t^*)\\
=\,\sum_{\genfrac{}{}{0pt}{1}{0\leq t_0<t^*<T}{t^*-t_0\leq k}}P_z(T_0=t_0,\tau_\d=t^*)
+\sum_{\genfrac{}{}{0pt}{1}{0\leq t_0<t^*<T}{t^*-t_0> k}}P_z(T_0=t_0,\tau_\d=t^*)
\quad\big(\dag \big)\,.
\end{multline*}
The first of the terms in the right--hand side 
of $(\dag)$
can be bounded thanks to the large deviations principle~\ref{pgdtransl}.
Indeed, if $0\leq t_0<t^*<T$
and $t^*-t_0<k$, we have
\begin{multline*}
P_z\big(
T_0=t_0,\tau_\d=t^*
\big)\,\leq\,
\sum_{z'\in\dU_\d(\rho^0)}P_z\big(
T_0=t_0,\tau_\d=t^*,Z_{t_0}=z'
\big)\\
\leq\,\sum_{z'\in\dU_\d(\rho^0)}P_{z'}\big(
Z_{t^*-t_0}\in\dU_\d(0)
\big)\,\leq\,m^{C(K)}\sup_{z'\in\dU_\d(\rho^0)}P_{z'}\big(
Z_{t^*-t_0}\in\dU_\d(0)
\big)\,,
\end{multline*}
where $C(K)$ is a positive constant that depends on $K$ but not on $m$.
We have then,
\begin{multline*}
\sum_{\genfrac{}{}{0pt}{1}{0\leq t_0<t^*<T}{t^*-t_0\leq k}}P_z\big(
T_0=t_0,\tau_\d=t^*
\big)\,\leq\,
T m^{C(K)}\sum_{h=0}^k\sup_{z'\in\dU_\d(\rho^0)}
P_{z'}\big(
Z_h\in\dU_\d(0)
\big)
\end{multline*}
Yet,
thanks to the large deviations principle of corollary~\ref{pgdtransl},
\begin{multline*}\limsup_{\genfrac{}{}{0pt}{1}{\ell,m\to\infty,\,q\to0}{\ell q\to a,\,m/\ell\to\a}}\,
\frac{1}{m}\ln 
\sum_{h=0}^k\sup_{z'\in\dU_\d(\rho^0)}P_{z'}\big(
Z_h\in\dU_\d(0)
\big)\\\leq\,-\min_{0\leq h\leq k}\inf\big\lbrace\,
V_h(x,y):x\in U_\d(\rho^0), y\in U_\d(0)
\,\big\rbrace\,\leq\,
-V^\d\,.
\end{multline*}
We deal next with the second term in the right--hand side of $(\dag)$.
Let us define the set $U$ to be the union of all 
the $\d$--neighborhoods of the fixed points $\rho^b$, $b\in I(A)$,
and the set $W$ to be this same union, but
where the neighborhoods of $\rho^0$ and $0$ have been left out, i.e.,
$$U\,=\,\bigcup_{b\in I(A)}U_\d(\rho^b)\,,\qquad
W\,=\,\bigcup_{b\in I(A)\setminus\lbrace\,0,\,K+1\,\rbrace}U_\d(\rho^b)\,.$$
We define the random time $T_1^*$ by
$$T_1^*\,=\,\min\big\lbrace\,
n\geq T_0:Z_n\in\W
\,\big\rbrace\,.\\
$$
We break the second term of the right--hand side of $(\dag)$ 
as follows,
\begin{multline*}
\sum_{\genfrac{}{}{0pt}{1}{0\leq t_0<t^*<T}{t^*-t_0>k}}P_z\big(
T_0=t_0,\tau_\d=t^*
\big)\,
=\,\sum_{\genfrac{}{}{0pt}{1}{0\leq t_0<t^*<T}{t^*-t_0>e^{\e m}}}P_z\big(
T_0=t_0,\tau_\d=t^*
\big)
\\
+\sum_{\genfrac{}{}{0pt}{1}{0\leq t_0<t^*<T}{k<t^*-t_0<e^{\e m}}}P_z\big(
T_0=t_0,\tau_\d=t^*, Z_t\not\in\dU,t_0<t<t^*
\big)
\\
+\sum_{\genfrac{}{}{0pt}{1}{0\leq t_0<t_1^*<t^*<T}{k<t^*-t_0<e^{\e m}}}P_z\big(
T_0=t_0,T_1^*=t_1^*,\tau_\d=t^*
\big)\quad\big(\ddag\big)\,.
\end{multline*}
The first of the sums in the right--hand side of $(\ddag)$ can be bounded thanks to lemma~\ref{timefar0rho}.
Indeed, if $t^*-t_0>e^{\e m}$,
then
\begin{multline*}
P_z\big(
T_0=t_0,\tau_\d=t^*
\big)\,
=\,\sum_{z'\in\dU_\d(\rho^0)}P_z\big(
T_0=t_0,Z_{t_0}=z',\tau_\d=t^*
\big)\\
\leq\,\sum_{z'\in\dU_\d(\rho^0)}P_{z'}\big(
Z_t\not\in\dU_\d(\rho^0)\cup\dU_\d(0),0\leq t\leq t^*-t_0
\big)\\
\leq\,
m^{C(K)}\sup_{z'\in\dU_\d(\rho^0)}P_{z'}\big(
Z_t\not\in\dU_\d(\rho^0)\cup\dU_\d(0),0\leq t\leq t^*-t_0
\big)\,\leq\,m^{C(K)}e^{-\e m/2}\,.
\end{multline*}
The second of the sums in the right--hand side of $(\ddag)$ can be bounded thanks to corollary~\ref{timefar2}.
Let $h$ and $c$ be as in corollary~\ref{timefar2}.
Then, we have, for $0\leq t_0<t^*<T$ and $t^*-t_0>k$,
\begin{multline*}
P_z\big(
T_0=t_0,\tau_\d=t^*,Z_t\not\in\dU, t_0<t<t^*
\big)\\
=\,\sum_{z'\in\dD\setminus\dU}P_z\big(
T_0=t_0,Z_{t_0+1}=z',\tau_\d=t^*,Z_t\not\in\dU, t_0<t<t^*
\big)\\
\leq\!\sum_{z'\in\dD\setminus\dU}P_{z'}\big(
Z_t\in\dD\setminus\dU,0\leq t<t^*-t_0
\big)\,
\leq\,
m^{C'(K)}\exp\bigg(
-m c\Big\lfloor
\frac{t^*-t_0}{h}
\Big\rfloor
\bigg),
\end{multline*}
where $C'(K)$ is a positive constant that depends on $K$ but not on $m$.
Thus,
\begin{multline*}
\sum_{\genfrac{}{}{0pt}{1}{0\leq t_0<t^*<T}{t^*-t_0>k}}P_z\big(
T_0=t_0,\tau_\d=t^*, Z_t\not\in\dU,t_0<t<t^*
\big)\\\leq\,T^2m^{C'(K)}\exp\bigg(
-m c\Big\lfloor
\frac{k}{h}
\Big\rfloor
\bigg)\,.
\end{multline*}
In order to bound the last sum in $(\ddag)$,
we introduce, for $b\in I(A)$,
the random time 
$$T^b_1\,=\,\sup\big\lbrace\,
n\leq t^*:Z_n\in\dU_\d(\rho^b)
\,\big\rbrace\,.$$
We decompose the last term in $(\ddag)$ as follows,
\begin{multline*}
\sum_{\genfrac{}{}{0pt}{1}{0\leq t_0<t_1^*<t^*<T}{k<t^*-t_0<e^{\e m}}}
P_z\big(
T_0=t_0,T_1^*=t_1^*,\tau_\d=t^*
\big)\,=\\
\sum_{\genfrac{}{}{0pt}{1}{b\in I(A)}{b\neq 0,K+1}}
\sum_{\genfrac{}{}{0pt}{1}{0\leq t_0<t_1^*<t_1<t^*<T}{k<t^*-t_0<e^{\e m}}}
P_z\big(
T_0=t_0,T_1^*=t_1^*,
Z_{t_1^*}\in\dU_\d(\rho^b),T_1^b=t_1,\tau_\d=t^*
\big)\,.
\end{multline*}
For a given $b\in I(A)\setminus\lbrace\,0,K+1\,\rbrace$,
we decompose further the above sum, by considering the three following cases:

$\bullet$ If $t_1^*-t_0>k$, then, 
for some positive constant $C(K)$
that depends on $K$ only, the above sum can be bounded by
$$Te^{3\e m}m^{C(K)}\exp\Bigg(
-m \bigg\lfloor
\frac{k}{h}
\bigg\rfloor
\Bigg)\,.$$
$\bullet$ If $t_1^*-t_0<k$ and $t^*-t_1<k$, then
the sum can be bounded thanks to the large deviations principle in corollary~\ref{pgdtransl},
which gives the following bound:
\begin{multline*}
Te^{\e m}\exp\bigg(
-m\Big(
\inf\big\lbrace\,
V(x,y):x\in U_\d(\rho^0),y\in U_\d(\rho^b)
\,\big\rbrace-\e
\Big)
\bigg)\times\\
\exp\bigg(
-m\Big(
\inf\big\lbrace\,
V(x,y):x\in U_\d(\rho^b),y\in U_\d(0)
\,\big\rbrace-\e
\Big)
\bigg)
<Te^{-m(V^\d-3\e)}\,.
\end{multline*}
$\bullet$ If $t_1^*-t_0<k$ and $t^*-t_1>k$,
then we define the set $W^{-b}$ by
$$W^{-b}\,=\,\bigcup_{\genfrac{}{}{0pt}{1}{b'\in I(A)}{b'\neq 0, b, K+1}}
U_\d(\rho^b)\,,$$
and we define the hitting time of $\W^{-b}$ after the time $T_1$ by
$$T_2^*\,=\,\inf\big\lbrace\,
n\geq T_1:Z_n\in\W^{-b}
\,\big\rbrace\,.$$
Then, the sum can be bounded by
\begin{multline*}
\sum_{\genfrac{}{}{0pt}{1}{0\leq t_0<t_1^*<t_1<t}{t^*-t_1>k,\, t^*-t_0<e^{\e m}}}
P_z\left(
\begin{matrix}
T_0=t_0,T_1^*=t_1^*,Z_{t_1^*}\in\dU_\d(\rho^b),\\
T_1=t_1,\tau_\d=t^*,Z_t\not\in\dU,t_1<t<t_1^*
\end{matrix}
\right)
\\+
\sum_{\genfrac{}{}{0pt}{1}{0\leq t_0<t_1^*<t_1<t_2^*<t}{t_1^*-t_0<k,\, t^*-t_0<e^{\e m}}}
P_z\left(
\begin{matrix}
T_0=t_0,T_1^*=t_1^*,Z_{t_1^*}\in\dU_\d(\rho^b),\\
T_1=t_1,T_2^*=t_2,\tau_\d=t^*
\end{matrix}
\right)
\end{multline*}
The first of the sums is again bounded by
$$Te^{3\e m}m^{C(K)}\exp\Bigg(
-m c\bigg\lfloor
\frac{k}{h}
\bigg\rfloor
\Bigg)\,.$$
In order to bound the second sum, we can break it again in three different cases,
and iterate this same procedure until we exhaust the fixed points in the set $I(A)$.
We will then get $3|I(A)|$ summands,
each of them being bounded by 
$$
\max\Bigg\lbrace\,
Te^{M(K)\e m}m^{C(K)}\exp\Bigg(
-m c\bigg\lfloor
\frac{k}{h}
\bigg\rfloor
\Bigg)\,,
Te^{-m(V^\d-M(K)\e)}
\,\Bigg\rbrace\,,
$$
where $M(K)$ is a natural number depending on $K$ only.
We choose $k$ large enough so that
$$c\Big\lfloor
\frac{k}{h}
\Big\rfloor\,>\,
\inf\big\lbrace\,
V(x,y):x\in U_\d(\rho^0), y\in U_\d(0)
\,\big\rbrace\,.$$
We set 
$$T\,=\,\exp\bigg(
m\Big(
\inf\big\lbrace\,
V(x,y):x\in U_\d(\rho^0),y\in U_\d(0)
\,\big\rbrace-\d
\Big)
\bigg)\,.$$
Then, taking $\e$ small enough so that $M(K)\e<\d$,
and putting the above estimates together, we conclude that, asymptotically,
$$P_z(\tau_\d\leq T)\,\leq\,e^{-m(V_\d-M(K)\e)}\,.$$
We deduce from here that,
for every $z\in\dU_\d(\rho^0)$,
$$
\liminf_{\ell,m,q}\,\frac{1}{m}\ln E_z(\tau_\d)
\geq\,\inf\big\lbrace\,
V(x,y):x\in U_\d(\rho^0),y\in U_\d(0)
\,\big\rbrace-\d\,.
$$
Now let $z\in\dD\setminus\lbrace\,0\,\rbrace$ 
and note that, from the proof of theorem~\ref{enoughms}
we can deduce that there exists $C>0$ such that
$$
P_z\big(
Z_{\lfloor C\ln m\rfloor}\in\dU_\d(\rho^0),
Z_t\neq 0, 0\leq t\leq \lfloor C \ln m\rfloor
\big)\,\geq\,e^{-\d m}\,.
$$
Therefore, for every $T\geq \lfloor C\ln m\rfloor$, 
we have
\begin{multline*}
P_z(\tau_0>T)\,\geq\,
\sum_{z'\in\dU_\d(\rho^0)}P_z\big(
\tau_0>T,Z_{\lfloor C\ln  m\rfloor}=z'
\big)\\
\geq\,\sum_{z'\in\dU_\d(\rho^0)}P_z\big(
Z_{\lfloor C\ln m\rfloor}=z',Z_t\neq 0,0\leq t\leq \lfloor C\ln m\rfloor
\big)\\
\times\P_{z'}\big(
\tau_0>T-\lfloor C\ln m\rfloor
\big)\,.
\end{multline*}
Thus, for any $z\in\dD\setminus\lbrace\,0\,\rbrace$
\begin{multline*}E_z(\tau_0)\,=\,\sum_{T\geq0}P_z(\tau_0>T)\,\geq\,
\sum_{T\geq \lfloor C\ln m\rfloor}P_z(\tau_0>T)\\
\geq\,\sum_{z'\in\dU_\d(\rho^0)}P_z\big(
Z_{\lfloor C\ln m\rfloor}=z',Z_t\neq 0,0\leq t\leq\lfloor C \ln m\rfloor
\big)\\
\times\,\sum_{T\geq \lfloor C\ln m\rfloor}P_{z'}\big(
\tau_0>T-\lfloor C\ln m\rfloor
\big)\\
\geq\,e^{-\e m}\inf_{z'\in\dU_\d(\rho^0)}E_{z'}(\tau_0)\,.
\end{multline*}
Yet, $\tau_\d<\tau_0$ by definition. Thus,
for any $z\in\dD\setminus\lbrace\,0\,\rbrace$,
$$\liminf_\lmq\,\frac{1}{m}\ln E_z(\tau_0)\,\geq\,\inf\big\lbrace\,
V(x,y):x\in U_\d(\rho^0),y\in U_\d(0)
\,\big\rbrace-2\d\,.$$
We let $\d$ go to zero and we get the desired result.
\end{proof}
In fact
the above calculations tell us further,
that for every $\d>0$, there exists $\e>0$ such that asymptotically,
$$\forall\,z\in\dD\setminus\lbrace\,0\,\rbrace\,,\qquad
P_z\big(
\tau_0\geq e^{m(V(\rho^0,0)-\d)}
\big)\,>\,1-e^{-\e m}\,.$$

\section{Concentration near \texorpdfstring{$\rho^0$}{rho0}}\label{Conc}
We assume throughout this whole section that $A(0)\exa>1$.
Our purpose is to study the behavior of the process $(Z_n)_{n\geq 0}$
inside the set $\dD$,
in order to show that it spends most of its time close to the fixed point $\rho^0$.
Let $\d>0$ and denote by $U_\d$ the $\d$--neighborhood of $\rho^0$, i.e.,
$$U_\d\,=\,\big\lbrace\,
r\in\cD:|r-\rho^0|_1<\d
\,\big\rbrace\,.$$
As previously, we write $\dU_\d\,=\,m U_\d\cap \dD$.
We also write $V(\rho^0,0)$ simply as $V$.
We introduce the following stopping times:
set $T_0=0$ and
\begin{align*}
T^*_{1}&=\inf\big\lbrace
n\geq T_{0}:
Z_n\in\dU_\d
\big\rbrace\qquad
&&T_{1}=\inf\big\lbrace
n\geq T^*_{1}:
Z_n\not\in\dU_{2\d}
\big\rbrace\\
&\ \, \vdots &&\ \ \quad\vdots\\
T^*_{i}&=\inf\big\lbrace
n\geq T_{i-1}:
Z_n\in\dU_\d
\big\rbrace\ 
&&T_{i}=\inf\big\lbrace
n\geq T^*_{i}:
Z_n\not\in\dU_{2\d}
\big\rbrace\\
&\ \, \vdots &&\ \ \quad\vdots
\end{align*}
Set also
\begin{align*}
\tau_0\,&=\,\inf\big\lbrace\,
n\geq0 : Z_n=0
\,\big\rbrace\,,\\
        \iota(n)\,&=\,\max\big\lbrace\,
i\leq n: T_{i-1}< n
\,\big\rbrace\,.
\end{align*}
Our purpose is to estimate the quantity
$$
E\bigg(
\sum_{i=1}^{\iota(\tau_0)}(T^*_i\wedge\tau_0-T_{i-1})
\bigg)\,.  $$
Noting that the argument in the expectation is bounded by $\tau_0$,
for any $i^*\in\N$, we can break the above expectation as follows:
$$E\bigg(
\sum_{i=1}^{\iota(\tau_0)}(T^*_i\wedge\tau_0-T_{i-1})
\bigg)\,\leq\,\sum_{i=1}^{i^*}
E\big(
1_{i\leq \iota(\tau_0)}(T_i^*-T_{i-1})
\big)+E(\tau_01_{\iota(\tau_0)>i^*})\,.
$$
If $1\leq i\leq \iota(\tau_0)$, then $T_{i-1}\leq \tau_0$ and $Z_{T_{i-1}}\neq 0$,
so that, using the Markov property,
$$
E\big(
1_{i\leq \iota(\tau_0)}(T_i^*-T_{i-1})
\big)\,\leq\,\sup_{z\in\dD\setminus\lbrace 0\rbrace}E_z\big(
T_1^*\wedge \tau_0
\big)\,.
$$
According to the corollary~\ref{enterneigh},
for any $\d,\e>0$, there exists $C=C(\d,\e)>0$,
such that asymptotically, for any $z\in\dD\setminus\lbrace\,0\,\rbrace$,
$$P_z\big(
Z_{\lfloor C\ln m\rfloor}\in\dU_\d
\big)\,\geq e^{-\e m}\,.$$
From this inequality we deduce the following bound.
\begin{corollary}
Let $\d,\e>0$. There exists $C=C(\d,\e)>0$ such that, asymptotically,
for every $z\in\dD\setminus\lbrace\,0\,\rbrace$,
$$\forall\,n\geq 0\qquad
P_z\big(
T_1^*\wedge\tau_0\geq n\lfloor C\ln m\rfloor
\big)\,\leq\,(1-e^{-\e m})^n\,.$$
\end{corollary}
The proof is similar to that of corollary~\ref{timefar2}.
Thanks to this bound, asymptotically,
for any $z\in\dD\setminus\lbrace\,0\,\rbrace$
\begin{multline*}
E_z\big(
T_1^*\wedge\tau_0
\big)\,=\,\sum_{k\geq 1}P_z\big(
T_1^*\wedge\tau_0\geq k
\big)\\
\leq\,\sum_{n\geq0}\sum_{k=n\lfloor C\ln m\rfloor+1}^{(n+1)\lfloor C\ln m\rfloor}
P_z\big(
T_1^*\wedge\tau_0\geq n\lfloor C\ln m\rfloor
\big)\\
\leq\,\lfloor C\ln m\rfloor\sum_{n\geq 0}(1-e^{-\e m})^n\,\leq\,e^{2\e m}\,.
\end{multline*}
We conclude that, for any $i^*\in\N$ and $\e>0$,
$$E\bigg(
\sum_{i=0}^{\iota(\tau_0)}(T_i^*\wedge\tau_0-T_{i-1})
\bigg)\,\leq\,
i^*e^{\e m}+E\big(
\tau_01_{\iota(\tau_0>i^*)}
\big)\,.$$
Let $\eta>0$ and define $t^\eta_m=e^{m(V+\eta)}$. Then,
\begin{align*}
E\big(
\tau_01_{\iota(\tau_0)>i^*}
\big)
\,&=\,E\big(
\tau_01_{\iota(\tau_0)>i^*}1_{\tau_0>t^\eta_m}
\big)+E\big(
\tau_01_{\iota(\tau_0)>i^*}1_{\tau_0\leq t^\eta_m}
\big)
\\&\leq\, 
E\big(
\tau_01_{\tau_0>t^\eta_m}\big)+t^\eta_m P\big(
\iota(t^\eta_m)>i^*
\big)
\end{align*}
Let us begin by bounding the first term on the right--hand side of this inequality.
We have, for every $n\in\N$ and $z\in\dD\setminus \lbrace\,0\,\rbrace$,
\begin{align*}
E_z\big(
\tau_01_{\tau_0>n}
\big)\,&=\,\sum_{k\geq 0}P_z\big(
\tau_01_{\tau_0>n}>k
\big)\\&=\,\sum_{k\geq0}P_z\big(
\tau_0>k\vee n
\big)\,\leq\,nP_z\big(
\tau_0>n
\big)+\sum_{k\geq n}P_z\big(
\tau_0>k
\big)\,.
\end{align*}
When proving the upper bound for the persistence time
(cf. the section~\ref{Pertime}),
we have shown the following inequality:
for every $\g>0$, there exists $C>0$ such that
$$\forall \,h\geq1\quad
\forall\,z\in\dD\setminus\lbrace\,0\,\rbrace\qquad
P_z\big(
\tau_0>h\lfloor C\ln m\rfloor
\big)\,\leq\,\Big(
1-e^{-m(V+\g)}
\Big)^h\,.$$
Using this inequality with $\g=\eta/2$, 
and setting $n=h\lfloor C\ln m\rfloor$, we get
\begin{multline*}
\!\!\!\!E_z\big(
\tau_01_{\tau_0>n}
\big)\,
\leq\,h\lfloor C\ln m\rfloor\Big(
1-e^{-m(V+\eta/2)}
\Big)^h+
\sum_{i\geq h}\lfloor C\ln m\rfloor P_z\big(
\tau_0>i\lfloor C\ln m\rfloor
\big)\\
\leq\, h\lfloor C\ln m\rfloor\Big(
1-e^{-m(V+\eta/2)}
\Big)^h+\lfloor C\ln m\rfloor\Big(
1-e^{-m(V+\eta/2)}
\Big)^he^{m(V+\eta/2)}\,.
\end{multline*}
Yet, if $h=\smash{\big\lfloor t^\eta_m/\lfloor C\ln m\rfloor\big\rfloor}$,
we have,
\begin{multline*}
E_z\big(
\tau_01_{\tau_0>t^\eta_m}
\big)\,\leq\,E_z\big(
\tau_01_{\tau_0>h\lfloor C\ln m\rfloor}
\big)\\
\leq\,e^{m(V+\eta)}\big(
1+\lfloor C\ln m\rfloor 
\big)\Big(
1-e^{-m(V-\eta/2)}
\Big)^h\\
\leq\,e^{m(V+\eta)}\big(
1+\lfloor C\ln m\rfloor
\big)e^{-h\exp(-mV-m\eta/2)}\,.
\end{multline*}
And since
$$he^{-m(V+\eta/2)}\,\geq\,
\bigg(
\frac{e^{m(V+\eta)}}{\lfloor C\ln m\rfloor}-1
\bigg)e^{-m(V+\eta/2)}\,\geq\,
\frac{e^{m\eta/2}}{\lfloor C\ln m\rfloor}-1\,,$$
the above expectation goes to $0$ when $m$ goes to infinity.
We deal now with the remaining term.
Let $\tau$ denote the exit time of the process $(Z_n)_{n\geq 0}$
from the set $\dU_{2\d}$, i.e.,
$$\tau\,=\,\inf\big\lbrace\,
n\geq 0: Z_n\not\in\dU_{2\d}
\,\big\rbrace\,.$$
We have the following bound on $\tau$.
\begin{lemma}
There exist $\g,\g'>0$ such that, asymptotically,
for all $z\in\dU_\d$,
$$P_z\big(
\tau\leq e^{m \g}
\big)\,<\,e^{-\g' m}\,.$$
\end{lemma}
\begin{proof}
Define $S$ to be the last time before $\tau$ that the process is in $\dU_\d$, i.e.,
$$S\,=\,\sup\big\lbrace\,
0\leq n\leq \tau:Z_n\in\dU_\d
\,\big\rbrace\,.$$
For any $n\geq 1$,
$$P_z\big(
\tau\leq n
\big)\,=\,\sum_{0\leq s<t\leq n}P_z\big(
S=s,\tau=t
\big)\,.$$
Let $h=h\geq 2$ and $c=c>0$ be as in corollary~\ref{timefar2}.
For a given value of $s$, we split the sum over $t$ in two parts:
$$
\sum_{0\leq s<t\leq n}P_z\big(
S=s,\tau=t
\big)\,=\,\sum_{t:t>s+h}\cdots\quad +\quad\sum_{t:s<t\leq s+h}\cdots\,.
$$
We study next the first sum, when $t>s+h$.
We condition on the state of the process at time $s+1$.
By the Markov property,
\begin{multline*}
\sum_{t:t>s+h}\cdots\,=\,\sum_{\genfrac{}{}{0pt}{1}{t:t>s+h}{z'\in\dU_{2\d}\setminus\dU_\d}}P_z\big(
S=s,Z_{s+1}=z',\tau=t
\big)\\
=\,\sum_{\genfrac{}{}{0pt}{1}{t:t>s+h}{z'\in\dU_{2\d}\setminus\dU_\d}}
P_z\left(
\begin{matrix}
S=s,Z_{s+1}=z',\\
Z_{s+1},\dots,Z_{t-1}\in\dU_{2\d}\setminus\dU_\d
\end{matrix}
\right)\\
=\!\!\!\sum_{\genfrac{}{}{0pt}{1}{t:t>s+h}{z'\in\dU_{2\d}\setminus\dU_\d}}\!\!\!P_z\big(
S=s,Z_{s+1}=z'
\big)
P_{z'}\big(
Z_1,\dots,Z_{t-s-2}\in\dU_{2\d}\setminus\dU_\d,
Z_{t-s-1}\not\in\dU_{2\d}
\big)\,.
\end{multline*}
Since the set $\dU_{2\d}\setminus\dU_\d$ contains none of the fixed points,
and since $t>s+h$,
by corollary~\ref{timefar2},
this last probability is smaller than $\exp(-mc\lfloor(t-s-2/h)\rfloor)$.
Therefore,
$$\sum_{t:t>s+h}\cdots\,\leq\,\sum_{t\geq h}
e^{-mc\left\lfloor\frac{t-1}{h}\right\rfloor}P_z\big(
S=s
\big)\,.$$
We bound next the second sum.
Conditioning on the state at time $s$:
\begin{multline*}
\sum_{t:s<t\leq s+h}\cdots\,=\,
\sum_{\genfrac{}{}{0pt}{1}{t:s<t\leq s+h}{z'\in\dU_\d}}P_z\big(
S=s,Z_s=z', \tau=t
\big)\\
\leq\,\sum_{\genfrac{}{}{0pt}{1}{t:s<t\leq s+h}{z'\in\dU_\d}}
P_{z'}\big(
Z_t\not\in\dU_{2\d}
\big)
P_z\big(
S=s,Z_s=z'
\big)\\
=\,\sum_{\genfrac{}{}{0pt}{1}{t:1\leq t\leq h}{z'\in\dU_\d}}
P_{z'}\big(
Z_t\not\in\dU_{2\d}
\big)
P_z\big(
S=s,Z_s=z'
\big)\,.
\end{multline*}
Using the large deviation principle of corollary~\ref{pgdtransl}, since $h$ is fixed, for any $t\leq h$
$$\limsup_\lmq\sup_{z'\in\dU_\d}P_{z'}\big(
Z_t\not\in\dU_{2\d}
\big)\,\leq\,-\inf\big\lbrace\,
V(x,y):x\in U_\d, y\not\in U_{2\d}
\,\big\rbrace\,.$$
Recall that $\d$ has been chosen small enough so that 
$G(\overline{U_\d})\subset U_\d$. Thus,
the above infimum is strictly positive.
We deduce that there exists $c'>0$ (depending on $d$) such that 
$$\sum_{t:1\leq t\leq h}
P_{z'}\big(
Z_t\not\in\dU_{2\d}
\big)\,\leq\,e^{-c' m}\,,$$
the bound being uniform over $z'\in\dU_\d$. 
\end{proof} 
Let us go back to the inequality
$$E\bigg(
\sum_{i=0}^{\iota(\tau_0)}(T_i^*\wedge\tau_0-T_{i-1})
\bigg)\,\leq\,i^*e^{\e m}+ E\big(
\tau_01_{\tau_0>t^{\eta}_m}
\big)+t^\eta_mP\big(
\iota(t^\eta_m)>i^*
\big)\,.$$
We set $i^*=2 e^{m(V+\eta-\g)}$.
Then, combining the previous lemma with the lemma~\ref{nbhits}, there exists $C>0$ such that
$$t^\eta_m
P\big(
\iota(t^\eta_m)>i^*
\big)\,=\,t^\eta_m P\big(
\iota(\frac{i^*}{2}e^{\g m})>i^*
\big)\,<\,t^\eta_me^{-(i^*-1)C}\,,$$
which goes to $0$ when $m$ goes to infinity.
We conclude, by choosing $\eta,\g,\e$ such that $\eta-\gamma+\e<-\e$, that
$$
E\bigg(
\sum_{i=0}^{\iota(\tau_0)}(T_i^*\wedge\tau_0-T_{i-1})
\bigg)\,\leq\,
e^{m(V-\e)}\,.
$$
\section{The neutral phase}\label{nphase}
The aim of this section is to study the process $(O_n)_{n\geq0}$
when none of the classes $0,\dots,K$ are present in the population.
Nevertheless, instead of using the occupancy process $(O_n)_{n\geq0}$
for our study, we will use a related process, namely the distance process $(D_n)_{n\geq0}$.
The distance process is a Markov chain on $\zl^m$; an element $d\in\zl^m$, is a vector that
represents the distances to the master sequence of the $m$ individuals present in the population.
The transition matrix $p_H$ of the distance process is given by
$$\forall\,d,e\in\zl^m\quad
p_H(d,e)\,=\,\prod_{1\leq i\leq m}\bigg(
\sum_{1\leq j\leq m}\frac{A(d(j))M(d(j),e(i))}{A(d(1))+\cdots+A(d(m))}
\bigg)\,.$$
The distance process and the occupancy process are related by a standard lumping procedure (cf. section 4 of~\cite{CerfWF}).
The distance process has been studied in detail in section~8 of~\cite{CerfWF}.
We state next some of the results therein, and we give a simple argument 
in order to obtain the remaining estimates that we will need.
Let $k\geq0$
We are interested in measuring the hitting time $\tau^*_k$ 
of the set of populations containing the classes $0,\dots,k$.
Let us define, with a slight abuse of notation,
\begin{align*}
\cW_k\,&=\,\big\lbrace\,
d\in\zl^m:d(i)\leq k\ \text{for some}\ 1\leq i\leq m
\,\big\rbrace\,,\\
	\cN_k\,&=\,\big\lbrace\,
d\in\zl^m:d(i)>k\ \text{for all}\ 1\leq i\leq m
\,\big\rbrace\,.
\end{align*}
The hitting time $\tau^*_k$ is then defined by
$$\tau^*_k\,=\,\inf\big\lbrace\,
n\geq0:D_n\in\cW_k
\,\big\rbrace\,.$$
The dynamics of the process $D_n$, started form any point in the set $\cN_K$,
and until the time $\tau^*_K$, is the same as if the fitness landscape were neutral.
Since we are ultimately interested in the hitting time $\tau^*_k$
for $k\geq K$, we will assume throughout the rest of this section that the 
fitness function $A$ is constant and equal to 1.

\textbf{Neutral hypothesis.} Throughout this section
we assume that $A(k)=1$ for all $k\geq0$.

Section~8 of~\cite{CerfWF} is concerned with estimating the hitting time $\tau^*_0$.
The main results therein that are of interest to us are contained in section~8.3;
we summarize them next.

$\bullet$ Concerning the expectation of the hitting time $\tau^*_0$,
for any $d\in\cN_0$,
$$\lim_\lmq \frac{1}{\ell}\ln E(\tau^*_0\,|\,D_0=d)\,=\,\ln\k\,.$$
$\bullet$ Concerning the concentration of $\tau^*_0$ around its mean,
for any $\e>0$ and $d\in\cN_0$,
$$\liminf_\lmq\frac{1}{\ell}\ln P(\tau^*_0>\k^{\ell(1-\e)}\,|\,D_0=d)\,=\,0\,.$$
Since for any $k\geq0$, the set $\cW_0$ is contained in the set $\cW_k$,
the hitting time $\tau_0$ must be larger than the hitting time $\tau_k$.
The following lemma is an immediate consequence of these observations.
\begin{lemma}\label{discup}
Asymptotically, for any $k\geq0$, $\e>0$ and $d\in\cN_k$,
$$E(\tau_k^*\,|\,D_0=d)\,\leq\,\k^{\ell(1+\e)}\,.$$
\end{lemma}
Our next purpose is to find a lower bound for $\tau_k^*$.
Consider a population $d\in\cD_k$;
there exists an individual at distance at most $k$ from the master sequence.
On one hand, the probability for this individual to be chosen for reproduction is bounded below by
$$\min_{0\leq i\leq k} A(i)/m A(0)\,.$$
On the other hand, the probability for this individual to transform into the master sequence by mutation
is bounded bellow (at least asymptotically) by
$$(1-q)^{\ell-k}\Big(
\frac{q}{\k-1}
\big)^k\,.$$
Combining these two facts, we deduce that, asymptotically, for any $d\in\cW_k$,
$$P(\tau^*_0=1\,|\,D_0=d)\,\geq\,\frac{\displaystyle \min_{0\leq i\leq k}A(i)}{mA(0)}
(1+q)^{\ell-k}\Big(
\frac{q}{\k-1}
\Big)^{k}\,\geq\,\ell^{-M}\,,$$
for some $M>0$ that depends on $k$ but not on $m,\ell,q$.
This inequality will be the key to bound the hitting time $\tau_k^*$.
Indeed, since every time we hit the set $\cW_k$, 
we have a fairly high probability of hitting the set $\cW_0$ in one step,
and since the hitting time $\tau_0^*$ is large with high probability,
the hitting time $\tau_k^*$ cannot be very small.
We formalize this idea in the following lemma:
\begin{lemma}\label{discdown}
For any $k\geq0$, $\e>0$ and $d\in\cN_k$,
$$\liminf_\lmq \frac{1}{\ell} \ln P(\tau_k^*>\k^{\ell(1-\e)}\,|\,D_0=d)\,=\,0\,.$$
\end{lemma}
\begin{proof}
Let $k\geq0$, $\e>0$ and let $d\in\cN_k$. Define $\d$ by 
$$-\d\,=\,\liminf_\lmq \frac{1}{\ell}\ln P(\tau_k^*>\k^{\ell(1-\e)}\,|\,D_0=d)\,.$$
Set $T=\k^{\ell(1-\e)}$ and let $N\in\N$. Asymptotically,
\begin{multline*}
P_d\big(\tau_0^*>N(T+1)\big)\,=\,P_d\big(\tau_k^*>T,\tau_0^*>N(T+1)\big)\\
+P_d\big(\tau_k^*\leq T,\tau^*_0>N(T+1)\big)\,.
\end{multline*}
The first of the terms is bounded by $\exp(-\d\ell/2)$, while the second one can be 
bounded by 
\begin{multline*}
\sum_{n\leq T}\sum_{e\in\cW_k}\sum_{f\not\in\cW_0}P_d\big(
\tau_k^*=n,D_n=e
\big)P_d\big(
D_{n+1}=f\,\big|\,\tau_1^*=n,D_n=e
\big)\\\times
P_d\big(
\tau_0^*>N(T+1)\,\big|\,\tau_k^*=n,D_n=e,D_{n+1}=f
\big)\,.
\end{multline*}
Yet, by the Markov property,
summing the second probability over $f$, we get
$$\sum_{f\not\in\cW_0}P_c\big(
D_{n+1}=f\,\big|\,\tau_1^*=n,D_n=e
\big)\,=\,P_e\big(\tau_0^*>1\big)\,\leq\,1-\ell^{-M}\,.$$
Using the Markov property again on the third probability, we conclude that
$$\sup_{d\in\cN_0}P_d\big(
\tau_0^*>N(T+1)
\big)\,\leq\,
e^{-\d\ell/2}+(1-\ell^{-M})\sup_{f\in\cN_0}P_f\big(
\tau_0^*>(N-1)(T+1)
\big)\,.
$$
Iterating this inequality, we obtain
\begin{multline*}
\sup_{d\in\cN_0}P_d\big(
\tau_0^*>N(T+1)
\big)\\\leq\,e^{-\d\ell/2}\sum_{n=0}^{N-1}(1-\ell^{-M})^n+(1-\ell^{-M})^N\,\leq\,\ell^Me^{-\d\ell/2} +(1-\ell^{-M})^N\,.
\end{multline*}
Thus, taking $0<\g<\e/4$ and letting $N=\k^{\ell\g}$,
we conclude from here that, asymptotically,
$$\sup_{d\in\cN_0}P_d\big(
\tau_0^*>\k^{\ell(1-\e/4)}
\big)\,\leq\,
e^{\d\ell/4}\,.
$$
Yet, in view of the result regarding the concentration of $\tau_0^*$ around its mean,
we must have $\d=0$.
\end{proof}

\section{The supercritical case}\label{super}
Define the function $a\mapsto\psi(a)$
to be equal to $V(\rho^0,0)$ on $\,]0,\ln A(0)[\,$, and 
to be equal to $0$ elsewhere.
We suppose that $\a\psi(a)>\ln\k$,
so that in particular, $A(0)\exa>1$, and $\rho^0$ is well defined.
Recall that the aim is to show that, for any 
continuous and bounded function $f:\R^{K+1}\lra\R$
$$\lim_\lmq\Bigg|\int_{\pml}f\bigg(
\frac{\pi_k(o)}{m}
\bigg)d\mu(o)-f(\rho^0)\Bigg|
\,=\,0\,.$$
Let $f:\R^{K+1}\lra\R$ be a continuous, bounded function.
By the ergodic theorem for Markov chains,
$$
\Bigg|\int_{\pml}f\bigg(
\frac{\pi_k(o)}{m}
\bigg)d\mu(o)-f(\rho^0)\Bigg|
\,\leq\,
\lim_{n\to\infty}\,
\frac{1}{n}\sum_{t=0}^{n-1}
\Bigg|
f\bigg(
\frac{\pi_K(O_t)}{m}
\bigg)-f(\rho^0)
\Bigg|\,.
$$
Let $\e>0$. We will prove that this last quantity
is smaller than $\e$, for $m,\ell$ large enough, $q$ small enough,
and $\ell q, m/\ell$ close enough to $a,\a$. 
We break the state space $\pml$
into two disjoint subsets,
the populations containing at least an individual
in one of the classes ${0,\dots,K}$,
$$\cW_K\,=\,
\Big\lbrace\,
o\in\pml:
o(0)+\cdots+o(K)\geq 1
\Big\rbrace\,,$$
and the populations containing no individuals
in any of the classes ${0,\dots,K}$,
$$\cN_K\,=\,
\Big\lbrace\,
o\in\pml:
o(0)+\cdots+o(K)=0
\Big\rbrace\,.$$
The process $\Ot$ will jump between these two sets.
We define the following sequence of stopping times, 
we set $\tau_0=0$ and 
\begin{align*}
\tau^*_1&=\inf\big\lbrace
n\geq0:
O_n\in\cW_K
\big\rbrace\qquad
&&\tau_1=\inf\big\lbrace
n\geq \tau^*_1:
O_n\in \cN_K
\big\rbrace\\
&\ \, \vdots &&\ \ \quad\vdots\\
\tau^*_k&=\inf\big\lbrace
n\geq \tau_{k-1}:
O_n\in \cW_K
\big\rbrace\ 
&&\tau_k=\inf\big\lbrace
n\geq \tau^*_k:
O_n\in\cN_K
\big\rbrace\\
&\ \, \vdots &&\ \ \quad\vdots
\end{align*}
Set next $\d>0$, and define the set $U_\d$ to
be the $\d$--neighborhood of $\rho^0$, i.e.,
$$U_\d\,=\,\big\lbrace\,
r\in\cD:|r-\rho^0|_1<\d
\,\big\rbrace\,.$$
The set $\cD$ being compact, the function $f$ is uniformly continuous on $\cD$.
We choose $\d$ small enough so that for every $r\in U_{2\d}$,
$$\big|f(r)-f(\rho^0)\big|\,<\,\e\,,$$
and so that the set $U_\d$ satisfies $G(\overline{U_\d})\subset U_\d$
(cf. theorem~\ref{convds}).
As in the previous sections,
for a set $A\subset \cD$ we denote by $\A$ the set $mA\cap \dD$.
For each $k\geq 0$ we define the following sequence of stopping times,
we set $T_{k,0}=\tau_k^*$ and
\begin{align*}
T^*_{k,1}&=\inf\big\lbrace
n\geq T_{k,0}:
Z_n\in\dU_\d
\big\rbrace\qquad
&&T_{k,1}=\inf\big\lbrace
n\geq T^*_{k,1}:
Z_n\not\in\dU_{2\d}
\big\rbrace\\
&\ \, \vdots &&\ \ \quad\vdots\\
T^*_{k,i}&=\inf\big\lbrace
n\geq T_{k,i-1}:
Z_n\in\dU_\d
\big\rbrace\ 
&&T_{k,i}=\inf\big\lbrace
n\geq T^*_{k,i}:
Z_n\not\in\dU_{2\d}
\big\rbrace\\
&\ \, \vdots &&\ \ \quad\vdots
\end{align*}
We distinguish between three different situations:
either $O_n$ is in $\cN_K$,
or $O_n$ is in $\cW$ and $\pi_K(O_n)$ is in a neighborhood of $\rho^0$,
or $O_n$ is in $\cW$ and $\pi_K(O_n)$ is outside a neighborhood of $\rho^0$.
We bound the above sum by breaking it according to these three situations,
which gives the following bound,
\begin{multline*}
\sum_{t=0}^{n-1}
\Bigg|
f\bigg(
\frac{\pi_K(O_t)}{m}
\bigg)-f(\rho^0)
\Bigg|\,\leq\,2||f||_\infty\sum_{k\geq 1}\big(
\tau^*_k\wedge n-\tau_{k-1}\wedge n
\big)
\\+\e n
+2||f||_\infty\sum_{k\geq 1}\sum_{i\geq 1}\big(
T^*_{k,i}\wedge \tau_k\wedge n-T_{k,i-1}\wedge\tau_k\wedge n
\big)
\,.\qquad (\clubsuit)
\end{multline*}
The next step is to bound the above sums.
We start with the first one of them.
We define, for $n\geq1$, the random variable $\iota(n)$ by 
$$\iota(n)\,=\,\max\big\lbrace\,  
k\geq0 : \tau_{k-1}< n
\,\big\rbrace\,.$$
We can rewrite the sum with the help of this new random variable as
$$\sum_{k\geq1}(\tau_k^*\wedge n-\tau_{k-1}\wedge n)\,=\,
\sum_{k=1}^{\iota(n)}(\tau_k^*\wedge n-\tau_{k-1})\,.$$
Define by $\tau(\cN_K)$ the hitting time of $\cN_K$, i.e.,
$$\tau(\cN_K)\,=\,\inf\big\lbrace\,
n\geq0 :O_n\in\cN_K
\,\big\rbrace\,.$$
By the last remark in the section~\ref{Pertime}, there exists a number $\g>0$
such that,
$$\max_{z\in\dU_\d}P_z\big(
\tau(\cN_K)\leq e^{m(V-\e)}
\big)\,<\,e^{-\g m}\,.$$
Thus, applying lemma~\ref{nbhits} with $N=e^{m(V-\e)}$ and $\l=1/2$,
it follows that, for every $h\geq 2$,
$$P\Big(
\iota\big(
\frac{h}{2}e^{m(V-\e)}
\big)\geq h
\Big)\,\leq\,e^{-(h-1)c}\,,$$
where $c$ is a positive constant, independent of $h$.
The next step is to bound the quantity 
$$E\Bigg(
\sum_{k=1}^{\iota(n)}(\tau^*_K\wedge n)-\tau_{k-1}
\Bigg)\,.$$
Let $i\geq 1$.
Since this quantity is obviously bounded by $n$,
we can decompose it according to whether $\iota(n)$ is greater or smaller than $i$ and bound it as follows,
$$E\Bigg(
\sum_{k=1}^{\iota(n)}(\tau^*_K\wedge n)-\tau_{k-1}
\Bigg)\,\leq\,nP\big(
\iota(n)\geq i
\big)+\sum_{k=1}^i E\big(
\tau^*_k-\tau_{k-1}
\big)\,.$$
In view of the lemma~\ref{discup}, asymptotically,
for every $o\in\cN_k$, we have 
$E_o(\tau^*_k-\tau_{k-1})\leq \k^{\ell(1+\e)}$;
we deduce that
$$E\Bigg(
\sum_{k=1}^{\iota(n)}(\tau^*_K\wedge n)-\tau_{k-1}
\Bigg)\,\leq\,nP\big(
\iota(n)\geq i
\big)+i\k^{\ell(1+\e)}\,.$$
Let us set 
$$i_n\,=\,\min\Big\lbrace\,
i:n\leq \frac{i e^{m(V-\e)}}{2}
\,\Big\rbrace\,.$$
On one hand, for $i=i_n$ we get
$$nP\big(
\iota(n)\geq i_n
\big)\leq\frac{i_ne^{m(v-\e)}}{2}P\bigg(\!
\iota\Big(
\frac{i_n e^{m(V-\e)}}{2}
\Big)\geq i_n\!
\bigg)\!\leq\frac{i_n e^{m(V-\e)}}{2}e^{-(i_n-1)\L^*(1/2)}.$$
This quantity goes to $0$ as $n$ goes to $\infty$. 
On the other hand,
$$\frac{1}{n}i_n\k^{\ell(1+\e)}\,\leq\,
\frac{2 i_n \k^{\ell(1+\e)}}{(i_n-1)e^{m(V-\e)}}\,.$$
When $n$ goes to $\infty$,
this last quantity converges to 
$$2\frac{\k^{\ell(1+\e)}}{e^{m(V-\e)}}\,=\,2\exp\Big(
m\big(
V-\e\frac{\ell}{m}(1+\e)\ln\k
\big)
\Big)\,,$$
which, for $\e$ small enough, goes to $0$ when $\ell,m$ go to $\infty$ and $q$ goes to $0$.
We proceed next to bound the second of the sums in $(\clubsuit)$.
For $n,k\geq 1$, we define the following random variables:
\begin{align*}
\iota^*(n)\,&=\,\max\big\lbrace\,
k\geq 0: \tau^*_k\leq n
\,\big\rbrace\,,\\
\iota_k(n)\,&=\,\max\big\lbrace\,
i\geq0: T_{k,i-1}< n
\,\big\rbrace\,
\end{align*}
We can rewrite the sum with the help of these 
new random variables as follows:
\begin{multline*}
\sum_{k\geq 1}\sum_{i\geq 1}\big(
T^*_{k,i}\wedge \tau_k\wedge n-T_{k,i-1}\wedge\tau_k\wedge n
\big)\,
=\,\sum_{k=1}^{\iota^*(n)-1}\Bigg(
\sum_{i=1}^{\iota_k(\tau_k)}
\big(
T_{k,i}^*\wedge\tau_k-T_{k,i-1}
\big)
\Bigg)
\\+\sum_{i=1}^{\iota_{\iota^*(n)}(\tau_{\iota^*(n)})}
\big(
T_{\iota_{\iota^*(n)},i}^*\wedge\tau_{\iota^*_n}\wedge n-T_{\iota_{\iota^*(n)},i-1}\wedge n
\big)\,.
\end{multline*}
Let $\iota(n)$ and $i_n$ be as in the previous section.
Taking the expectation, the above sum can be bounded by
$$nP\big(
\iota^*(n)\geq i_n
\big)+\sum_{k=1}^{i_n}E\Bigg(
\sum_{i=1}^{\iota_k(\tau_k)}\big(
T^*_{k,i}\wedge \tau_k-T_{k,i-1}
\big)
\Bigg)\,.$$
Noting that $\iota^*(n)\leq \iota(n)$, the first term can be shown to converge to $0$
as $n$ goes to $\infty$, as in the previous section.
Let us deal with the expectation.
We introduce the following stopping times: set $T_0=0$ and
\begin{align*}
T^*_{1}&=\inf\big\lbrace
n\geq T_{0}:
Z_n\in\dU_\d
\big\rbrace\qquad
&&T_{1}=\inf\big\lbrace
n\geq T^*_{1}:
Z_n\not\in\dU_{2\d}
\big\rbrace\\
&\ \, \vdots &&\ \ \quad\vdots\\
T^*_{i}&=\inf\big\lbrace
n\geq T_{i-1}:
Z_n\in\dU_\d
\big\rbrace\ 
&&T_{i}=\inf\big\lbrace
n\geq T^*_{i}:
Z_n\not\in\dU_{2\d}
\big\rbrace\\
&\ \, \vdots &&\ \ \quad\vdots
\end{align*}
Set also
\begin{align*}
\tau_0\,&=\,\inf\big\lbrace\,
n\geq0 : Z_n=0
\,\big\rbrace\,,\\
\iota(n)\,&=\,\max\big\lbrace\,
i\leq n: T_{i-1}< n
\,\big\rbrace\,.
\end{align*}
Fix $k\in\lbrace\,1,\dots,i_n\,\rbrace$.
By the Markov property,
\begin{multline*}
E\Bigg(
\sum_{i=1}^{\iota_k(\tau_k)}\big(
T^*_{k,i}\wedge \tau_k-T_{k,i-1}
\big)
\Bigg)\\
=\,\sum_{z\in\dD\setminus\lbrace 0\rbrace}
E\Bigg(
\sum_{i=1}^{\iota_k(\tau_k)}\big(
T^*_{k,i}\wedge \tau_k-T_{k,i-1}
\big)\,\Bigg|\,Z_{\tau_k^*}=z
\Bigg)P\big(
Z_{\tau_k^*}=z
\big)
\\
\leq\,\sup_{z\in\dD\setminus\lbrace 0\rbrace}
E_z\Bigg(
\sum_{i=1}^{\iota(\tau_0)}\big(
T^*_{i}\wedge \tau_0-T_{i-1}
\big)
\Bigg)\,.
\end{multline*}
Yet, as we have shown in section~\ref{Conc}, the last expectation is bounded by $e^{m(V-\g)}$,
for any $\g>0$.
Therefore,
$$\frac{1}{n}\sum_{1\leq k\leq i_n}E\Bigg(
\sum_{i=1}^{\iota_k(\tau_k)}\big(
T_{k,i}^*\wedge\tau_k-T_{k,i-1}
\big)
\Bigg)\,\leq\,\frac{i_n}{n}e^{m(V-\g)}\,\leq\,
\frac{2 e^{m(V-\g)}}{e^{m(V-\e)}}\,,$$
which, choosing $\e<\g$,
converges to $0$ when $m$ goes to $\infty$.

\section{The subcritical case}\label{sub}
We suppose that $\a\psi(a)<\ln\k$.
Recall that the aim is to show that, for any 
continuous and bounded function $f:\R^{K+1}\lra\R$
$$\lim_\lmq\Bigg|\int_{\pml}f\bigg(
\frac{\pi_k(o)}{m}
\bigg)d\mu(o)-f(\rho^0)\Bigg|
\,=\,0\,.$$
Let $f:\R^{K+1}\lra\R$ be a continuous, bounded function.
By the ergodic theorem for Markov chains,
$$
\Bigg|\int_{\pml}f\bigg(
\frac{\pi_k(o)}{m}
\bigg)d\mu(o)-f(0)\Bigg|
\,\leq\,
\lim_{n\to\infty}\,
\frac{1}{n}\sum_{t=0}^{n-1}
\Bigg|
f\bigg(
\frac{\pi_K(O_t)}{m}
\bigg)-f(0)
\Bigg|\,.
$$
Proceeding as in the subcritical case,
we obtain the following bound:
$$
\sum_{t=0}^{n-1}
\Bigg|
f\bigg(
\frac{\pi_K(O_t)}{m}
\bigg)-f(0)
\Bigg|\,\leq\,2||f||_\infty\bigg(\sum_{k= 1}^{\iota(n)-1}(
\tau_k-\tau_k^*)+n-\tau^*_{\iota(n)}\bigg)
\,.
$$
Denote by $\tau(\cW_K)$ the hitting time of the set $\cW_K$, i.e.,
$$\tau(\cW_K)\,=\,\big\lbrace\,
n\geq0: O_n\in\cW_K
\,\big\rbrace\,.$$
In view of lemma~\ref{discdown}, for every $\e,\g>0$, we have
$$\max_{c\in\cN_K}P_z\big(
\tau(\cW_K)>\k^{\ell(1-\e)}
\big)\,\geq\,e^{-\g m}\,.$$
Thus, using lemma~\ref{nbhits} with $\l=1-e^{-\g m}/2$ and $N=\k^{\ell(1-\e)}$,
we conclude that, for all $h\geq 2$,
$$P\Big(
\iota\big(
h(1-e^{-\g m}/2)\k^{\ell(1-\e)}
\big)\geq h
\big)\,<\,e^{-(h-1)c}\,,$$
where $c$ is a positive constant which does not depend on $h$.
The next step is to bound the quantity
$$
E\Bigg(
\sum_{k=1}^{\iota(n)-1}\big(
\tau_k-\tau_k^*
\big)
+n-\tau_{\iota(n)}^*
\Bigg)\,.
$$
Let $i\geq 1$.
Since this quantity is obviously bounded by $n$,
we can decompose it as according to whether $\iota(n)$
is greater or smaller than $i$ and bound it as follows
$$E\Bigg(
\sum_{k=1}^{\iota(n)-1}\big(
\tau_k-\tau_k^*
\big)
+n-\tau_{\iota(n)}^*
\Bigg)\,\leq\,
nP\big(
\iota(n)\geq i
\big)+
\sum_{k=1}^{i}E\big(
\tau_k-\tau_k^*
\big)\,.$$
Since for every $o\in\cW$
$$E\big(\tau_1\,\big|\,O_0=o\big)\,\leq\,
\exp\big(
m(V+\e)
\big)\,,$$
we deduce that
$$E\Bigg(
\sum_{k=1}^{\iota(n)-1}\big(
\tau_k-\tau_k^*
\big)
+n-\tau_{\iota(n)}^*
\Bigg)\,\leq\,
nP\big(
\kappa(n)\geq i
\big)+
i\exp\Big(
m\big(
\psi(a)+\e
\big)
\Big)\,.$$
Let us set 
$$i_n\,=\,\min\big\lbrace\,
i:n\leq i (1-e^{-\g m}/2)\k^{\ell(1-\e)}
\,\big\rbrace\,.$$
On one hand, for $i=i_n$ we get
\begin{multline*}
nP\big(
\iota(n)\geq i_n
\big)\,\leq\,
i_n (1-e^{-\g m}/2)\k^{\ell(1-\e)}\\
\times P\Big(
\iota\big(i_n(1-e^{-\g m}/2)\k^{\ell(1-\e)}\big)\geq i_n
\Big)\,\leq\,
i_n (1-e^{-\g m}/2)\k^{\ell(1-\e)}e^{-i_n c}\,.
\end{multline*}
This quantity goes to $0$ as $n$ goes to infinity.
On the other hand,
$$\frac{i_n}{n}
e^{m(V+\e)}\,\leq\,
\frac{i_n}{(i_n-1)(1-e^{-\g m}/2) \k^{\ell(1-\e)}}
e^{m(V+\e)}\,.$$
When $n$ goes to infinity,
this last quantity converges to
\begin{multline*}
\frac{e^{m(V+\e)}}{(1-e^{-\g m}/2)\k^{\ell(1-\e)}}\\=\,
\exp\Big(
-\ell\big(
\ln\k-\a\psi(a)-\e-\e\ln\k+(\ln (1-e^{-\g m}/2))/\ell
\big)
\Big)\,,
\end{multline*}
which, for $\e$ small enough, 
goes to $0$ with $\ell,m,q$.

\bibliographystyle{plain}
\bibliography{wfart}
\appendix
\section{Properties of the mutation matrix}
\label{Propmut}
We give here the most relevant properties of the lumped 
mutation matrix $(M_H(i,j),0\leq i,j\leq \ell)$.
The $i$--th row of the lumped mutation matrix
is given by the different of two independent binomial
laws, i.e.,
if $X\sim Bin(i,q/(\k-1))$ and $Y\sim Bin(\ell-i,q)$
are independent random variables, then
$$M_H(i,j)\,=\,P(i-X+Y=j)\,.$$
Fix $i$ and $j$
and let $\ell$ go to infinity, $q$ go to 0,
and $\ell q$ go to $a$;
the first of the binomial laws converges to a Dirac mass at $0$,
while the second one convergence to a Poisson random variable
of parameter $a$. Thus,
$$M_\infty(i,j)\,=\,\lim_{\genfrac{}{}{0pt}{1}{\ell\to\infty,\,q\to0}{\ell q\to a}}
M_H(i,j)\,=\,\begin{cases}
\exa\displaystyle\frac{a^{j-i}}{(j-i)!} &\text{if}\ j\geq i\,,\\
0 &\text{otherwise}\,.
\end{cases}$$
In particular, in the limit, there is no back mutation.
Furthermore, for $\ell$ large enough, $q$ small enough,
and $\ell q$ close enough to $a$,
$$\forall\,i>j\qquad
M_H(i,j)\,\leq\,M_H(j+1,j)\,.$$

\section{Bounds on hitting times}
Let $E$ be a finite set and $(X_n)_{n\geq 0}$ a recurrent Markov chain on $E$.
For a set $A\subset E$ we denote by $\tau_A$ the hitting time of $A$,
i.e.,
$$\tau_A\,=\,\inf\big\lbrace\,n\geq 0:X_n\in A\,\big\rbrace\,.$$
Let $A\subset B\subset E$ and
define the following sequence of stopping times, 
we set $T_0=0$ and 
\begin{align*}
T^*_1&=\inf\big\lbrace
n\geq0:
X_n\in A
\big\rbrace\qquad
&&T_1=\inf\big\lbrace
n\geq T^*_1:
X_n\not\in B
\big\rbrace\\
&\ \, \vdots &&\ \ \quad\vdots\\
T^*_k&=\inf\big\lbrace
n\geq T_{k-1}:
X_n\in A
\big\rbrace\ 
&&T_k=\inf\big\lbrace
n\geq T^*_k:
X_n\not\in B
\big\rbrace\\
&\ \, \vdots &&\ \ \quad\vdots
\end{align*}
Define, for $n\geq1$, the random variable $\iota(n)$ by 
$$\iota(n)\,=\,\max\big\lbrace\,  
k\geq0 : T_{k-1}< n
\,\big\rbrace\,.$$
Our objective is to give a bound 
on the random variable $\iota(n)$.
Let us assume that there exist $N,p>0$ such that
$$\max_{z\in A}\,P\big(
\tau_{E\setminus B}\leq N\,\big|\, X_0=z
\big)\,<\,p\,.$$
\begin{lemma}\label{nbhits}
For any $h\geq 1$ and $\l> p$,
there exists $c>0$ (depending on $\l$ but not on $h$), such that
$$P\Big(
\iota\big(h\l N \big)\geq h
\Big)\,<\,e^{-(h-1)c}\,.
$$
\end{lemma}
\begin{proof}
Let us assume that $h\l$ is an integer number (otherwise we may replace it by $\lfloor h \l\rfloor$).
From the definition of $\iota(n)$, we see that
$$\iota\Big(
h\l N\Big)\,\geq\, h\quad\Leftrightarrow\quad T_{h-1}\,<\,h\l p\,.$$
We define the random variables $(Y_i)_{i\geq 1}$ by setting
$$Y_i\,=\,T_i-T^*_i\,,\quad i\geq1\,.$$
Then,
$$T_{h-1}\,\geq\,Y_1+\cdots+Y_{h-1}\,.$$
In view of the assumption on $\tau_{E\setminus B}$, for every $i\geq 1$,
$$P\big(
Y_i\leq N
\big)\,<\,p\,.$$
We define the following sequence of Bernoulli random variables
$$\e_i\,=\,1_{Y_i\leq N}\,,\quad i\geq 1\,.$$
Thus, if $T_{h-1}<h\l N$,
at least $(h-1)\l$ of the random variables $Y_1,\dots,Y_{h-1}$ must satisfy
$Y_i\leq N$. Whence,
$$P\big(
T_{h-1}<h\l N\big)\,\leq\,
P\big(
\e_1+\cdots+\e_{h-1}\geq (h-1)\l
\big)\,.$$
We use the exponential Chebyshev inequality in order to bound the last probability:
for any $\beta>0$ we have
$$P\big(
\e_1+\cdots+\e_{h-1}\geq (h-1)\l
\big)\,\leq\,e^{-\b\l}E\big(
e^{\b\e_1/(h-1})\cdots e^{\b\e_{h-1}/(h-1)}
\big)\,.$$
The random variables $\e_1,\dots,\e_{h-2}$ are measurable with respect to
$$\smash{\big(X_n,0\leq n\leq T^*_{h-1}\big)}$$.
Thus, thanks to the strong Markov property,
\begin{multline*}
E\big(
\e^{\b\e_1/(h-1)}\cdots e^{\b\e_{h-1}/(h-1)}
\big)\\
=\,E\Big(
E\big(
e^{\b\e_1/(h-1)}\cdots e^{\b\e_{h-1}/(h-1)}\,\big|\,
X_0,\dots,X_{T^*_{h-1}}
\big)
\Big)\\
=\,E\Big(
e^{\b\e_1/(h-1)}\cdots e^{\b\e_{h-2}/(h-1)}E\big(
e^{\b\e_{h-1}/(h-1)}\,\big|\,X_0,\dots,X_{\tau^*_{h-1}}
\big)
\Big)\,.
\end{multline*}
Yet, for all $x\in A$,
$$E\big(
e^{\b\e_1/(h-1)}\,\big|\,X_0=z
\big)\,\leq\,e^{\b/(h-1)}p+1-p\,.$$
Iterating, this procedure, we obtain
$$E\big(
e^{\b\e_1/(h-1)}\cdots e^{\b\e_{h-1}/(h-1)}
\big)\,\leq\,\big(
e^{\b/(h-1)}p+1-p
\big)^{h-1}\,.$$
We make the change of variables $\b\to(h-1)\b$ in order to obtain
\begin{multline*}
P\big(
\e_1+\cdots+\e_{h-1}\geq (h-1)\l
\big)\,\leq\,\exp\Big(
-(h-1)\big(\b\l-\ln(
e^{\b}p+1-p
)
\big)
\Big)\,.
\end{multline*}
Denote by $\L^*(t)$ the Cramèr transform of the Bernoulli law with parameter $p$,
$$
\L^*(t)\,=\,\sup_{\b\geq0}\big(
\b t-\ln(
e^\b p+1-p
)
\big)\,
=\,t\ln\frac{t}{p}+(1-t)\ln\frac{1-t}{1-p}\,.
$$
Optimizing the previous inequality over $\b$,
we obtain
$$P\big(
\e_1+\cdots+\e_{h-1}\geq (h-1)\l
\big)\,\leq\,\exp\big(
-(h-1)\L^*(\l)
\big)\,,$$
where $\L^*(\l)>0$ is independent of $h$.
It follows that
$$P\big(
\iota(
h\l N
)\geq h
\big)\,\leq\,e^{-(h-1)\L^*(\l)}\,,$$
as wanted.
\end{proof}
\end{document}